\documentclass[11pt]{article}

\usepackage[latin1]{inputenc}
\usepackage{epsfig}
\usepackage{graphicx,psfrag}
\usepackage[tight]{subfigure} 

\usepackage{amsfonts,amsthm,bbm,amssymb,epsf,amsmath,%
latexsym,amscd,amsfonts,enumerate,amsthm,supertabular,color,url}

\newcommand{\RR}{\mathbb{R}}

\DeclareMathOperator{\diag}{diag}

\newtheorem{theorem}{Theorem}
\newtheorem{lemma}[theorem]{Lemma}
\newtheorem{corollary}[theorem]{Corollary}
\newtheorem{proposition}[theorem]{Proposition}
\newtheorem{definition}[theorem]{Definition}

\newtheorem{remark}[theorem]{Remark}

\newcommand{\ZZ}{\mathbb{Z}}
\newcommand{\NN}{\mathbb{N}}

\def\Im{{\rm Im}\;\!}

\def\c{{\bf c}}

\def\bA{{\bf A}}
\def\b{{\bf b}}
\def\d{{\bf d}}

\def\z{{\bf z}}
\def\w{{\bf w}}

\def\R{{\bf R}}

\def\x{{\bf x}}

\def\y{{\bf y}}

\def\u{{\bf u}}
\def\X{{\bf X}}
\def\Y{{\bf Y}}
\def\bv{{\bf v}}
\def\oo{{\bf 0}}
\def\bomega{{\boldsymbol\omega}}
\def\btheta{{\boldsymbol\theta}}

\def\cL{\mathcal{L}}

\newcommand{\dnorm}{\|\hspace{0.5pt}{\cdot}\hspace{0.5pt}\|}

\def\bdelta{{\boldsymbol\delta}}

\newcommand{\red}[1]{{\color{red} #1 }}    %
\renewcommand{\red}[1]{#1}

\begin{document}

\title{Minimal Numerical Differentiation Formulas}%

\author{Oleg Davydov\thanks{Department of Mathematics, University of Giessen, Arndtstrasse 2, 
35392 Giessen, Germany, \tt{oleg.davydov@math.uni-giessen.de}}
\and Robert Schaback\thanks{Institut für Numerische und Angewandte Mathematik,
Universit\"at G\"ottingen, Lotzestra\ss{}e 16--18, 37083 G\"ottingen, Germany,
{\tt schaback@math.uni-goettingen.de}}}%

\maketitle

\begin{abstract}
We investigate numerical differentiation formulas on irregular centers in two or more variables that are exact for polynomials of a given order
and minimize an absolute seminorm of the weight vector. Error bounds are given in terms of a growth function that carries the information about
the geometry of the centers. Specific forms of weighted $\ell_1$ and weighted least squares minimization are proposed that produce numerical 
differentiation formulas with particularly good performance in numerical experiments. 
\red{The results are of interest in particular for meshless generalized finite difference methods as they provide a consistency error analysis 
for such methods.}
\end{abstract}

\section{Introduction}
We consider a 
linear differential operator $D$ of order $k$ in $d$ real variables in the notation
\begin{equation}\label{Dop}  
Df=\sum_{\alpha\in\mathbb{Z}_+^d\atop|\alpha|\le k}c_\alpha\partial^\alpha f,
\qquad
\partial^\alpha:=\frac{\partial^{|\alpha|}}{\partial \x^{\alpha}}=
\frac{\partial^{|\alpha|}}{\partial x_1^{\alpha_1}\cdots\partial x_d^{\alpha_d}},
\qquad|\alpha|=\alpha_1+\cdots+\alpha_d,
\end{equation}
where $c_\alpha$ are real functions of the independent variable.
To simplify notation, we assume that $d\ge1$ and $k\ge0$ are fixed for 
the entire paper, and do not
indicate
the dependence
of various quantities on these two parameters.

Given an operator $D$, a point $\z\in\RR^d$  such that
$
\sum_{|\alpha|= k}|c_\alpha(\z)|\ne 0
$, 
and a finite set $\X=\{\x_1,\ldots,\x_N\}\subset\RR^d$, we  consider 
numerical differentiation formulas  
\begin{equation}\label{ndif} 
D f(\z)\approx\sum_{j=1}^{N} w_j\, f( \x_j )
\end{equation}
that allow to evaluate $D f(\z)$ approximately using 
only  function values at the centers in $\X$. Any formula of type \eqref{ndif} is defined by its 
weight vector $\w=[w_1,\ldots,w_N]^T$ that depends on $D$, $\z$ and $\X$.

Numerical differentiation formulas can be used in particular in meshless 
\emph{generalized finite difference methods}
for partial differential equations that differ from the classical finite difference method in
that they replace finite differences on grids by numerical
differentiation formulas \eqref{ndif} on subsets of a given set of centers that discretize a
spatial domain (see for example 
\cite{FFprimer15,Ostermann_et_al13}).
In these methods $N$ is usually kept small, and 
a refinement of the solution requires increasing the density of the centers and hence reducing the
distances $\|\z-\x_j\|_2$. The accuracy of the approximation \eqref{ndif} plays a
crucial role for the convergence of such methods. It has been thoroughly investigated under the name of
\emph{local discretization error} or \emph{consistency error} for the functions of one variable (with applications
to numerical methods for ordinary differential equations) and for many variables in the case when the
centers $\z,\x_1,\ldots,\x_N$ are placed on a grid. 
However, little is known about the error of the numerical differentiation formulas
when the centers are irregularly distributed in $\RR^d$, $d\ge2$, the
situation of particular interest for the meshless methods, 
\red{since these often employ the use of a (random) subset of neighboring
  points to attain system sparsity}.

In this paper  we are interested in the error bounds of the form
\begin{equation}\label{consist}
\Big|D f(\z)-\sum_{j=1}^{N} w_j\, f( \x_j )\Big|\le \sigma(\z,\X)\phi(h_{\z,\X})\|f\|_F,\quad f\in F,
\end{equation}
where
\begin{equation}\label{hzXdef}
h_{\z,\X}:=\displaystyle{  \max_{\x\in\X}\|\z-\x\|_2},
\end{equation}
$F$ is a space of functions with a (semi)norm $\dnorm_F$, $\sigma(\z,\X)$ depends on the geometry of 
$\X$ and its position with respect to $\z$, but not on $h_{\z,\X}$, and $\phi(h_{\z,\X})$ stands for the local approximation order, where
$\phi(t)\to0$ as $t\to+0$. If $\phi(h_{\z,\X})=\mathcal{O}(h_{\z,\X}^\mu)$ for some $\mu>0$, then 
\eqref{ndif} has \emph{consistency order $\mu$} in the usual sense of the error analysis of the 
numerical methods as treated for example in \cite{Stetter73}.
Together with the computable estimates of a stability constant of the system matrix
suggested in \cite{Schaback_error_analysis17}, these bounds can be used for the practical
error analysis of generalized finite difference methods.

Given a point set $\X=\{\x_1,\ldots,\x_N\}$, a numerical differentiation formula 
\eqref{ndif} can be generated by
requiring exactness for polynomials of certain  order. 
We denote by $\Pi^d_{q}$ the space of all $d$-variate polynomials of order 
at most $q$, i.e.\ of total degree less than $q$, and $\Pi^d_{0}:=\{0\}$.  %

\begin{definition}\label{polycons}\rm 
Let $D$ be a linear differential operator
of order $k$. A numerical differentiation formula (\ref{ndif}) 
based on a set $\X=\{\x_1,\ldots,\x_N\}\subset\RR^d$ %
is said to be \emph{(polynomially) exact of order $q\ge1$} if
\begin{equation}\label{qexact} 
D p(\z)=\sum_{j=1}^{N} w_j\, p( \x_j ) \quad\hbox{for all}\quad  p\in \Pi^d_{q},
\end{equation}
and \emph{(polynomially) consistent of order $m\ge1$} 
if it is exact of order $q=m+k$. %
\end{definition}

Polynomially exact formulas of order $q$ can be obtained by 
solving the equations \eqref{qexact} with respect to the weights $w_j$.
 The achievable order of polynomial exactness on a set $\X$
is limited by solvability of (\ref{qexact})
and depends crucially on the geometry of $\X$ and the differential operator
$D$. If \eqref{qexact} admits more than one solution, then the remaining degrees of 
freedom can be settled by minimizing a norm of the weight vector $\w=[w_1,\ldots,w_N]^T$.
For example, the minimization of the $\ell_2$-norm $\|\w\|_2:= (\sum_{j=1}^{N} w_j^2)^{1/2}$ results in the weight 
vector obtained by applying $D$ to the least squares polynomial fit to the data, see Section~\ref{ell2}.
Weight vectors that minimize the weighted 
$\ell_1$-norm $\|\w\|_{1,\mu}=\sum_{j=1}^{N} |w_j|\|\x_j-\z\|_2^{\mu}$  and satisfy a positivity condition 
were considered in \cite{Seibold08},
whereas
\cite{DavySchaback16} gives an error bound for the numerical differentiation formulas that minimize 
$\|\w\|_{1,q}$.

In this paper we develop error bounds for the \emph{$\dnorm$-minimal formulas} whose weights minimize a 
given  absolute seminorm $\dnorm$. These error bounds are
expressed in terms of a \emph{growth function}
$$
\rho_{q,D}(\z,\X,\dnorm):=\sup\{D p(\z):p\in \Pi_q^d,\;\|p|_\X\|^*\leq 1\},
\quad \dnorm^*\;\text{is dual seminorm,}$$
and are based on a duality argument that shows that
$$
\rho_{q,D}(\z,\X,\dnorm)=\inf\Big\{\|\w\|:\w\in\RR^N,\;D p(\z)=\sum_{j=1}^Nw_jp(\x_j)
\text{ for all }p\in\Pi_q^d\Big\}.$$
Special types of growth functions have previously been considered in 
\cite{davydov:2007-1,beatson-et-al:2010-1,DavySchaback16}. In particular, they appear in the error bounds for
the kernel-based numerical differentiation in \cite{DavySchaback16}.

Special attention is given to  $\dnorm_{1,\mu}$- and $\dnorm_{2,\mu}$-minimal formulas, where
$\|\w\|_{2,\mu}:=(\sum_{j=1}^{N} w_j^2\|\x_j-\z\|_2^{2\mu})^{1/2}$. The $\dnorm_{1,\mu}$-minimal formulas are 
sparse but their computation requires linear programming whereas the $\dnorm_{2,\mu}$-minimal formulas are 
cheaper to compute thanks to their relation to the weighted least squares. 
Theoretical estimates and numerical experiments reveal a lot of similarity in the degree of accuracy and stability of both 
types of formulas for the same $\mu$, which makes it possible to access the quality  of a  $\dnorm_{1,\mu}$-minimal 
formula if the weights of the respective $\dnorm_{2,\mu}$-minimal formula are known, see Remark~\ref{rho2}. The experiments suggest that the
preferable choice of the exponent $\mu$ is $\mu\approx q$, where $q$ is the exactness order, as it provides formulas with almost optimal accuracy and
and good stability even on centers with difficult geometry. 
Explicit error bounds are obtained for the positive $\dnorm_{1,\mu}$-minimal formulas of \cite{Seibold08}.

We refer the reader to our paper \cite{DavySchabackOPT} for a treatment of the \emph{polyharmonic} formulas that 
minimize the error of a polynomially consistent formula in a Beppo-Levi space and share 
with the $\dnorm$-minimal formulas the computationally advantageous property of \emph{scalability}, see Remark~\ref{scalable} and
Section~\ref{compute}.

The paper is organized as follows. In Section 2 we develop error bounds of the type \eqref{consist} for arbitrary polynomially consistent
formulas, with $f$ in a Hölder space 
$C^{r,\gamma}(\Omega)$
or  a Sobolev space $W^r_\infty(\Omega)$. Section 3 is devoted to the duality theory of the growth functions
and error bounds for the $\dnorm$-minimal formulas, Section 4 to the $\dnorm_{1,\mu}$-minimal and Section 5 to the $\dnorm_{2,\mu}$-minimal 
formulas. Numerical experiments are  presented in Section 6 \red{and some concluding remarks on possible applications in Section 7}.

\section{Error of Polynomially Consistent Formulas}\label{error}
We start with a basic error bound generalizing \cite[Theorem 10]{DavySchaback16}.
For any domain $\Omega\subset\RR^d$, 
$r\in\mathbb{Z}_+$ and $\gamma\in(0,1]$, let $C^{r,\gamma}(\Omega)$ denote
the Hölder space consisting of all $r$ times continuously differentiable functions $f$ on $\Omega$
 such that $|\partial^\alpha f|_{\gamma,\Omega}<\infty$ for all 
$\alpha\in\mathbb{Z}_+^d$ with $|\alpha|=r$, where
$$
|f|_{\gamma,\Omega}:=\sup_{\x,\y\in\Omega\atop \x\ne\y}\frac{|f(\x)-f(\y)|}{\|\x-\y\|_2^\gamma}$$
is a semi-norm on $C^{0,\gamma}(\Omega)$. For $r\ge1$ we use a nonstandard semi-norm
\begin{equation}\label{Hoe}
|f|_{r,\gamma,\Omega}:=\frac{1}{(\gamma+1)\cdots (\gamma+r)}
\Big(\sum_{|\alpha|=r}{r\choose \alpha}|\partial^\alpha f|_{\gamma,\Omega}^2\Big)^{1/2},
\quad f\in C^{r,\gamma}(\Omega).
\end{equation}
Note that \eqref{Hoe} is equivalent to 
the usual seminorm $|f|_{C^{r,\gamma}(\Omega)}:=\max_{|\alpha|=r}|\partial^\alpha f|_{\gamma,\Omega}$ since
$$
\frac{1}{(r+1)!}\max_{|\alpha|=r}|\partial^\alpha f|_{\gamma,\Omega}\le|f|_{r,\gamma,\Omega}<\frac{d^{r/2}}{r!}\max_{|\alpha|=r}|\partial^\alpha f|_{\gamma,\Omega}.$$
In the case $\gamma=1$ the space $C^{r,1}(\Omega)$ can be identified with the Sobolev space
$W^{r+1}_\infty(\Omega)$ with the semi-norm 
$$
|f|_{W^{r+1}_\infty(\Omega)}:=\max_{|\alpha|=r+1}\|\partial^\alpha f\|_{L^\infty(\Omega)}$$
if $\Omega$ satisfies a strong local Lipschitz condition, see for example \cite[Lemma 4.28]{Adams03}.
If $\Omega$ is a convex domain, it can be easily verified that
\begin{equation}\label{Cr}
|f|_{r,1,\Omega}\le |f|_{\infty,r+1,\Omega}:=\frac{1}{(r+1)!}\Big(\sum_{|\alpha|=r+1}{r+1\choose \alpha}
\|\partial^\alpha f\|_{L^\infty(\Omega)}^2\Big)^{1/2}
\end{equation}
for any $f\in W^{r+1}_\infty(\Omega)$.

Given $\z\in\RR^d$ and a function $f\in C^{s-1}(\Omega)$ for some domain $\Omega\subset\RR^d$ containing $\z$, 
we denote by $T_{s,\z}f$ the  Taylor polynomial of order $s\ge1$  centered at $\z$,
\begin{equation*}%
T_{s,\z}f(\x)=\sum_{|\alpha|<s}\frac{(\x-\z)^\alpha}{\alpha!}\partial^\alpha f(\z).
\end{equation*}
For any $\z\in\RR^d$ and $\X=\{\x_1,\ldots,\x_N\}\subset\RR^d$ we set
\begin{equation}\label{SzX}
S_{\z,\X}:=\bigcup_{i=1}^N[\z,\x_i],
\end{equation}
where $[\x,\y]$ denotes the segment
$\{\alpha\x+(1-\alpha)\y:
0\le\alpha\le 1\}$.

\begin{lemma}\label{Taylor} 
Assume that $f\in C^{r,\gamma}(\Omega)$, $r\ge0$, $\gamma\in(0,1]$. Then for any $\z,\x\in\Omega$ such that 
$[\z,\x]\subset\Omega$,
$$%
|f(\x)-T_{r+1,\z}f(\x)|\le \|\x-\z\|_2^{r+\gamma}|f|_{r,\gamma,\Omega}.
$$%
\end{lemma}

\begin{proof} 
By a well-known remainder formula, for any $f\in C^{r}(\Omega)$ and $\x\in\Omega$,
\begin{equation*}%
R_{r,\z}f(\x):=f(\x)-T_{r,\z}f(\x)=r\sum_{|\alpha|=r}\frac{(\x-\z)^\alpha}{\alpha!}
\int_0^1(1-t)^{r-1}\partial^\alpha f(\z+t(\x-\z))\,dt.
\end{equation*}
Hence
\begin{equation}\label{Te}
R_{r+1,\z}f(\x)=r\sum_{|\alpha|=r}\frac{(\x-\z)^\alpha}{\alpha!}
\int_0^1(1-t)^{r-1}[\partial^\alpha f(\z+t(\x-\z))-\partial^\alpha f(\z)]\,dt,
\end{equation}
and we have for $f\in C^{r,\gamma}(\Omega)$,
\begin{align*}%
|R_{r+1,\z}f(\x)| &\le r\sum_{|\alpha|=r}\frac{|(\x-\z)^\alpha|}{\alpha!}
\int_0^1(1-t)^{r-1}t^\gamma\|\x-\z\|_2^{\gamma}|\partial^\alpha f|_{\gamma,\Omega}\,dt \\
&=\frac{r!}{(\gamma+1)\cdots (\gamma+r)}\|\x-\z\|_2^{\gamma}
\sum_{|\alpha|=r}\frac{|(\x-\z)^\alpha|}{\alpha!}|\partial^\alpha f|_{\gamma,\Omega}.
\end{align*}
By using the identity 
\begin{equation*}%
\sum_{|\alpha|=r}\frac{(\x-\z)^{2\alpha}}{\alpha!}=\frac{1}{r!}\|\x-\z\|_2^{2r}
\end{equation*}
that follows from the multinomial theorem, and applying Cauchy-Schwarz inequality,
we obtain
\begin{align*}
\Big(\sum_{|\alpha|=r}\frac{|(\x-\z)^\alpha|}{\alpha!}|\partial^\alpha f|_{\gamma,\Omega}\Big)^2
&\le 
\frac{1}{r!}\|\x-\z\|_2^{2r}
\sum_{|\alpha|=r}\frac{|\partial^\alpha f|_{\gamma,\Omega}^2}{\alpha!},
\end{align*}
and the statement follows.
\end{proof}

\begin{proposition}\label{Thepb} Any differentiation formula (\ref{ndif}), which is
exact of order $q>k$ for a linear differential operator $D$ of order $k$, has an error bound 
\begin{equation}\label{peb}
|D f(\z)-\sum_{j=1}^{N} w_j\, f( \x_j )|\le
|f|_{q-1,\gamma,\Omega}
\sum_{j=1}^{N} |w_j|\|\x_j-\z\|_2^{q-1+\gamma}
\end{equation}
for all $f\in C^{q-1,\gamma}(\Omega)$, $\gamma\in(0,1]$, where $\Omega\subset\RR^d$ is any domain that 
contains the set $S_{\z,\X}$.  
\end{proposition}

\begin{proof} %
Using the exactness of the formula and the fact that $Df(\z)=D T_{q,\z}(\z)$ for any differential operator $D$ of order less than $q$, we 
obtain
\begin{align}\label{rem}
D f(\z)-\sum_{j=1}^{N} w_j\, f( \x_j )
&=-\sum_{j=1}^{N} w_jR_{q,\z}f( \x_j ),
\end{align}
and \eqref{peb} follows from  Lemma~\ref{Taylor}.
\end{proof}

By using \eqref{Cr}, we obtain the following formulation  for functions in $W^{q}_\infty(\Omega)$. 

\begin{proposition}\label{Copb} Any differentiation formula (\ref{ndif}), which is
exact of order $q>k$ for a linear differential operator $D$ of order $k$, has an error bound 
\begin{equation}\label{Cpeb}
|D f(\z)-\sum_{j=1}^{N} w_j\, f( \x_j )|\le
|f|_{\infty,q,\Omega}
\sum_{j=1}^{N} |w_j|\|\x_j-\z\|_2^{q}
\end{equation}
for all $f\in W^{q}_\infty(\Omega)$, where $\Omega\subset\RR^d$ is any domain that 
contains the convex hull of $\{\z,\x_1,\ldots,\x_N\}$.  
\end{proposition}

Note that \cite[Theorem~10]{DavySchaback16} gives the same estimate \eqref{Cpeb}
under the
stronger assumption that 
$f\in C^{q}(\Omega)$. 

\begin{remark}\label{remTe}\rm
The expression of the error of numerical differentiation in the form \eqref{rem} with 
$R_{q,\z}f$ given by \eqref{Te} generalizes \cite[Theorem 1]{CR1972} that applies to the weights $w_j$ obtained by
differentiating Lagrange basis polynomials in the case when $\X$ is suitable for interpolation with polynomials in $\Pi^d_q$.
An error bound in a form similar to \eqref{Cpeb}  for the weights generated by polynomial interpolation or
least squares polynomial fits can be found in \cite{CSV08b}. 
\end{remark}

Looking for an estimate of the type \eqref{consist}, we can for example deduce from \eqref{Cpeb} 
the bound
\begin{equation}\label{Cpebs}
|D f(\z)-\sum_{j=1}^{N} w_j\, f( \x_j )|\le
\|\w\|_1h_{\z,\X}^{q}\,|f|_{\infty,q,\Omega}, \quad f\in W^{q}_\infty(\Omega),
\end{equation}
that holds under the hypotheses of Proposition~\ref{Copb}, where the $\ell_1$-norm of $\w$,
$$
\|\w\|_1=\sum_{j=1}^{N} |w_j|,$$
has an additional interpretation as a \emph{stability constant} of the numerical differentiation formula \eqref{ndif} responsible in
particular for the sensitivity of \eqref{ndif} to the round-off errors because the absolute error in
$\sum_{i=1}^Nw_jf_j$ is bounded by $\|\w\|_1$ times the maximum absolute error in $f_j$.

However, the factor $\|\w\|_1$ in \eqref{Cpebs} also depends on $h_{\z,\X}$ as the following lemma shows.

\begin{lemma}\label{stabestb} Let (\ref{ndif}) be exact of order $q> k$ for a linear differential 
operator $D$ of order $k$. 
There is a constant $C$ depending only on $D$ and $\z$, such that
$$
\|\w\|_1\ge Ch_{\z,\X}^{-k}.$$
\end{lemma}

\begin{proof}
Since $\sum_{|\alpha|= k}|a_\alpha(\z)|\ne0$, there is $\alpha_0$ with $|\alpha_0|= k$ such that
$a_{\alpha_0}(\z)\ne0$. Then for $p(\x)=(\x-\z)^{\alpha_0}$, we have 
$|p(\x_j)|\leq \|\x_j-\z\|_2^k\le h_{\z,\X}^k$, $j=1,\ldots,N$, and $Dp(\z)=\alpha_0! a_{\alpha_0}(\z)$, where
$p\in\Pi^d_{k+1}$. Since (\ref{ndif}) is exact for $p$, it follows that
$$
|D p(\z)|=\Big|\sum_{j=1}^{N} w_j\, p( \x_j )\Big|\le \|\w\|_1\max_{j=1,\ldots,N}|p(\x_j)|,$$
and hence 
$\|\w\|_1\ge\alpha_0! |a_{\alpha_0}(\z)| h_{\z,\X}^{-k}$.
\end{proof}

This motivates to express the bounds in terms of the quantity $h_{\z,\X}^{k}\|\w\|_1$,
that is the $\ell_1$-norm of the scaled vector $h_{\z,\X}^{k}\w$.
More general, we define $\sigma$-factors of the form
\begin{equation}\label{sig}
\sigma(\z,\X,\w,\mu):=h_{\z,\X}^{k-\mu}\sum_{j=1}^{N} |w_j|\|\x_j-\z\|_2^{\mu},\quad \mu>0,
\quad \sigma(\z,\X,\w,0):=h_{\z,\X}^{k}\|\w\|_1.
\end{equation}
A simple calculation shows that
\begin{equation}\label{sigs}
\sigma(\z,\X,\w,\mu)\le \sigma(\z,\X,\w,\nu),\quad 0\le\nu<\mu.
\end{equation}

An important feature of the expressions $\sigma(\z,\X,\w,\mu)$, $\mu\ge0$, is their 
\emph{scale-invariance} in the sense that
$$ 
\sigma(\z,\X,\w,\mu)=\sigma(\z,\X^h,\w^h,\mu)\quad\text{for any}\; h>0,
$$ 
where $\X^h=\{\x^h_j:j=1,\ldots,N\}$ and $\w^h=[w^h_1,\ldots,w^h_N]^T$ with 
\begin{equation}\label{Xhwh}
\x^h_j:=\z+h(\x_j-\z),\quad w^h_j:=h^{-k}w_j,\quad j=1,\ldots,N.
\end{equation}
The scaling of the weights in \eqref{Xhwh} is standard in the finite difference method and is justified
by the fact that for a homogeneous operator $D$ of order $k$ (that is $a_\alpha=0$
for all $\alpha$ with $|\alpha|<k$ in \eqref{Dop}), the exactness order of all scaled formulas 
\begin{equation}\label{ndifsc}
(D f)(\z)\approx\sum_{j=1}^{N} w^h_j\, f( \x^h_j ),\quad h>0,
\end{equation}
coincides with the exactness order of \eqref{ndif}.

We deduce from Propositions~\ref{Thepb} and \ref{Copb} the following main result of this section.

\begin{theorem}\label{Thw1k}
Assume that the differentiation formula (\ref{ndif})  
for a linear differential operator $D$ of order $k$ is exact of order $q>k$, 
and let $\Omega$ be a domain containing the set $S_{\z,\X}=\bigcup_{i=1}^N[\z,\x_i]$. Then 
 for any $r=k,\ldots,q-1$, $\gamma\in(0,1]$, and $\mu\le r+\gamma$,
\begin{equation}\label{t1kbH}
|D f(\z)-\sum_{j=1}^{N} w_j\, f( \x_j )|\le
\sigma(\z,\X,\w,\mu)\,h_{\z,\X}^{r+\gamma-k}
\,|f|_{r,\gamma,\Omega},\quad f\in C^{r,\gamma}(\Omega).
\end{equation}
If $\Omega$ contains the convex hull of $\{\z,\x_1,\ldots,\x_N\}$ and $\mu\le r+1$, then
\begin{equation}\label{t1kbW}
|D f(\z)-\sum_{j=1}^{N} w_j\, f( \x_j )|\le
\sigma(\z,\X,\w,\mu)\,h_{\z,\X}^{r+1-k}
\,|f|_{\infty,r+1,\Omega},\quad f\in W^{r+1}_\infty(\Omega).
\end{equation}
\end{theorem}

The estimates in \eqref{t1kbH} and \eqref{t1kbW} are of the type \eqref{consist}
discussed in
the introduction.
The factors in the right hand side are responsible for three different ingredients of the error:
geometric position of the points ($\sigma$-factor), 
size of the domain of influence ($h$-factor), and a measure of smoothness of the function ($f$-factor).

\section{Minimal Formulas and Growth Functions}\label{gfun}

If the weights $w_j$ are not fully determined by the exactness 
condition \eqref{qexact}, it is natural in view of Theorem~\ref{Thw1k} 
to choose $\w$ that  minimizes the factor $\sigma(\z,\X,\w,\mu)$, that is to
minimize the \emph{weighted $\ell_1$-(semi)norm} 
\begin{equation}\label{wl1norm}
\sum_{j=1}^{N} |w_j|\|\x_j-\z\|_2^{\mu}=:\|\w\|_{1,\mu} 
\end{equation}
of the weight vector $\w=[w_1,\ldots,w_N]^T$, with an appropriate exponent $\mu>0$. 
Clearly, $\dnorm_{1,\mu}$ is just a seminorm in case that $\x_j=\z$ for some $j$.
Note that the
notation $\|\w\|_{1,\mu}$ for the above seminorm hides its dependence on $\z,\X$, which normally does not
cause any confusion. %

Since, however, 
weight vectors minimizing other expressions are also in use, we note that %
\begin{equation}\label{pebpw}
\sum_{j=1}^{N} |w_j|\|\x_j-\z\|_2^{\mu}\le \|\w\|\,\|\bdelta^{\mu}(\z,\X)\|^*,
\end{equation}
where $\dnorm$ is an arbitrary absolute seminorm on $\RR^N$, $\dnorm^*$ its dual seminorm, and 
$\bdelta^\mu(\z,\X)$
denotes the vector with components
\begin{equation}\label{delta}
\delta^\mu_j(\z,\X):=%
\|\x_j-\z\|_2^\mu,\quad j=1,\ldots,N.
\end{equation}
Recall that a (semi)norm $\dnorm$ on $\RR^N$ is said to be \emph{absolute} if it 
depends only on the absolute values of the components of a vector.
The arguments in \cite{BSW61} (originally for norms on $\mathbb{C}^N$) show that any absolute seminorm 
is \emph{monotonic} in the sense that $|w_j|\le |v_j|$, $j=1\ldots,N$, implies $\|\w\|\le\|\bv\|$.
From this it is easy to deduce that the kernel $\mathcal{K}$ of an absolute seminorm $\dnorm$ has the form 
$\{\w\in\mathbb{R}^N:w_j=0,\;j\in J\}$ for some $J\subset\{1,\ldots,N\}$. The
dual
seminorm
$\dnorm^*$ is
the norm on the orthogonal complement $\mathcal{K}^\perp$ of  $\mathcal{K}$ dual to the restriction of
$\dnorm$ to $\mathcal{K}^\perp$, and  $\|\u\|^*=\infty$ for all 
$\u\in\mathbb{R}^n\setminus \mathcal{K}^\perp$. In particular, 
\begin{equation}\label{dwl1norm}
\|\u\|_{1,\mu}^*=\max_{j=1,\ldots,N\atop u_j\ne 0} |u_j|\|\x_j-\z\|_2^{-\mu},
\quad \u\in\mathbb{R}^N,
\end{equation}
and $\|\bdelta^{\mu}(\z,\X)\|_{1,\mu}^*=1$.

If we introduce a basis of polynomials, we see that polynomial consistency
needs weight vectors $\w\in\RR^N$ satisfying a 
possibly underdetermined linear system $\bA\w=\b$ 
in the form (\ref{qexact}) with a $M\times N$ matrix $\bA$. Then one can try to minimize 
$\|\w\|$ for any norm or seminorm on $\RR^N$ under all solutions.

\begin{lemma}\label{Lemdualgen}
The problems
$$
\inf\{\|\w\|\;:\; \w\in\RR^N,\;\bA\w=\b  \} \;\hbox{ and }\;
\sup\{\b^T\x\;:\; \x\in\RR^M,\; \|\bA^T\x\|^*\leq 1\} 
$$
are dual, if $\dnorm^*$ is the dual seminorm to $\dnorm$. In particular,
\begin{enumerate}
 \item $\b^T\x\leq \|\w\|$ holds for admissible vectors $\w$ and all $\x\in\RR^M$
 such that $\|\bA^T\x\|^*\le1$, and hence both problems are solvable
if one is,
\item if both problems are solvable, then optimal values are equal,
\item the second is unbounded if and only if the first has no admissible vector.
\end{enumerate} 
\end{lemma} 
\begin{proof}
Although this lemma is a special case of Fenchel's duality theorem, we provide a short
self-contained proof for the case when $\dnorm$ is a norm. (The proof for seminorms can be obtained
following the same lines.)
For any admissible vector $\w$ and any $\x\in\RR^M$ satisfying $\|A^T\x\|^*\le1$, we have
$$%
\b^T\x=(\bA\w)^T\x=\w^T\bA^T\x\le \|\w\|\|A^T\x\|^*\le \|\w\|,
$$ %
which proves the first assertion.

To prove the second assertion, we assume that the first problem is solvable and
thus $\b\in\Im \bA$. Then the formula
$$
\lambda(\bA^T\x):=\b^T\x,\quad \x\in\R^M,$$
defines a linear functional $\lambda$ on the space $\Im \bA^T\subset\RR^N$. 
Indeed, if $\bA^T\x=\bA^T\y$, then $\x-\y\in\ker \bA^T$ and hence 
$\b^T(\x-\y)=0$ by Fredholm's alternative for matrices, which shows that
$\lambda$ is well defined.
If we equip $\Im \bA^T$ with the dual norm $\dnorm^*$, then the norm of  
$\lambda$,
$$
\|\lambda\|=\sup_{\|\bA^T\x\|^*\leq 1}\b^T\x$$
is the optimal value of the second problem.
By the Hahn-Banach theorem the functional $\lambda$ can be extended to $\RR^N$ without
increasing its norm. Hence there exists $\w\in\RR^N$ such that 
$\|\w\|=\|\lambda\|$ and $\w^T\bA^T\x=\b^T\x$ for all $\x\in\RR^M$. The latter
property implies $\bA\w=\b$, which shows that $\w$ is an admissible vector for the
first problem and
$$\|\w\|= \sup_{\|\bA^T\x\|^*\leq 1}\b^T\x.$$

The third assertion follows from Fredholm's alternative
that ensures that for any $\b\notin\Im \bA$ there is 
$\x\in \ker \bA^T$  such that $\b^T\x\ne0$.
\end{proof}

Note that the dual problem in Lemma \ref{Lemdualgen} always has admissible points
but may be unbounded. 

Inspecting the dual problem in our case, we see that the minimization
of the seminorm $\|\w\|$ of a weight vector $\w$ for a 
differentiation formula which is exact on $\Pi_q^d$ is dual to
the maximization of $Dp(\z)$ over all polynomials $p\in \Pi_q^d$
where the vector $p|_\X$ 
of values of $p$ on $\X$ satisfies $\|p|_\X\|^*\leq 1$.    %
This generalizes directly to other linear functionals than $f\mapsto Df(\z)$.
By appropriate interpretation of $\bA$ and $\b$
 of Lemma \ref{Lemdualgen} we get 

\begin{theorem}\label{Thegro}
Given a seminorm $\dnorm$ %
on $\RR^N$, a set $\X\subset\RR^d$ of $N$ points,
and a linear functional $\lambda$ that can be applied to $\Pi_q^d$, the quantity
$$
\rho:=\sup\{\lambda(p):p\in \Pi_q^d,\;\|p|_\X\|^*\leq 1\} 
$$
is finite if and only if there exists an approximation formula
$$
\lambda(f)\approx \sum_{j=1}^Nw_jf(\x_j)
$$
which is exact on $\Pi_q^d$. Then 
$$
\rho=\inf\Big\{\|\w\|:\lambda(p)=\sum_{j=1}^Nw_jp(\x_j)
\text{ for all }p\in\Pi_q^d\Big\}.
$$
Moreover, the inequality
$$
|\lambda(p)|\leq \|\w\|%
$$
holds for all $\w$ that define a $\Pi_q^d$-exact formula and all
$p\in \Pi_q^d$ such that $\|p|_\X\|^*\le1$ and
is sharp for any $p=p^*$ and $\w=\w^*$ satisfying $\|\w^*\|=\rho=|\lambda(p^*)|$.\qed
\end{theorem}

In the case $\lambda(f):=\delta_\z Df=D f(\z)$ we call the constant $\rho$ of the last
theorem the \emph{growth function} and introduce the notation
\begin{equation}\label{gf}
\rho_{q,D}(\z,\X,\dnorm):=\sup\{D p(\z):p\in \Pi_q^d,\;\|p|_\X\|^*\leq 1\}.
\end{equation}
Theorem~\ref{Thegro} implies the following dual form of the growth function 
\begin{equation}\label{dual}
\rho_{q,D}(\z,\X,\dnorm)=\inf\Big\{\|\w\|:\w\in\RR^N,\;D p(\z)=\sum_{j=1}^Nw_jp(\x_j)
\text{ for all }p\in\Pi_q^d\Big\}.
\end{equation}

Note that the growth function was first considered in \cite{davydov:2007-1} for   the norm $\dnorm_1$ 
and identity operator $Df=f$ (in which case it is related to  the norming constant
of \cite{jetter-et-al:1999-1}), and later in \cite{beatson-et-al:2010-1} for partial derivatives and
in \cite{DavySchaback16} for differential operators and the norm $\dnorm_{1,\mu}$. 
In particular, different proofs of the duality \eqref{gf}--\eqref{dual} in special cases were given in 
\cite{davydov:2007-1} and \cite{beatson-et-al:2010-1}.

It is easy to see that the growth function is \emph{monotone with respect to $\X$} in the following sense. 
If $\X'=\X\cup\{\x_{N+1},\ldots,\x_{M}\}$, with $\x_j\notin\X$, $j=N+1,\ldots,M$, then  
any vector $\w\in\RR^N$ can be extended to $\w'\in\RR^M$ by setting $\w'_j=0$, $j=N+1,\ldots,M$.
Assuming that the seminorm $\dnorm'$ on $\RR^M$ is compatible with $\dnorm$ in the sense that $\|\w'\|'=\|\w\|$, we derive 
from \eqref{dual} that
\begin{equation}\label{gmon}
\rho_{q,D}(\z,\X',\dnorm')\le \rho_{q,D}(\z,\X,\dnorm).
\end{equation}

If $\X=\{x_1,\ldots,x_N\}$ is  a \emph{unisolvent set} for $\Pi_q^d$ in the
sense that $p|_\X=0$ implies $p=0$  for any $p\in\Pi_q^d$, then $\X$ contains an
\emph{interpolation set} $\X'=\{x'_1,\ldots,x'_M\}\subset\X$ for $\Pi_q^d$, that is a unisolvent set
such that $M=\dim \Pi_q^d={d+q-1\choose d}$. Then any polynomial 
$p\in\Pi_q^d$ satisfies $p=\sum_{j=1}^M  p(x'_j)\ell_j$, where the \emph{Lagrange
polynomials} $\ell_j\in\Pi_q^d$ are uniquely determined by the conditions
$\ell_i(x'_j)=\delta_{ij}$ (the Kronecker delta). 
Since $Dp(\z)=\sum_{j=1}^M D\ell_j(\z)\, p(x'_j)$, the set in the right hand
side of \eqref{dual} is not empty for any $\z\in\RR^d$, which implies the
following statement.

\begin{proposition}\label{unisolv}
If $\X$  is  a unisolvent set for $\Pi_q^d$, then 
$\rho_{q,D}(\z,\X,\dnorm)<\infty$ for any $D$, $\z$ and $\dnorm$.
\end{proposition}

However, $\X$ does not have to be a unisolvent set for $\Pi_q^d$ in order that
$\rho_{q,D}(\z,\X,\dnorm)$ $<\infty$ for certain points $\z$. An important example is given 
in the case of the Laplace operator 
$\Delta=\sum_{i=1}^d\frac{\partial^2}{\partial x_i^2}$ 
in $\RR^2$ by the \emph{five point star} 
$\X=\{(0,0),(\pm 1,0),(0,\pm 1)\}\subset\RR^2$, where
$\Delta p (0,0)=p(1,0)+p(-1,0)+p(0,1)+p(0,-1)-4p(0,0)$ for all $p\in\Pi^2_3$, which implies
$\rho_{3,\Delta}(0,\X,\dnorm)<\infty$ for any norm $\dnorm$, whereas $\X$ with its 5 points 
is too small to be a unisolvent set for $\Pi^2_3$ of dimension 10.

If $\X$ is an interpolation set for $\Pi_q^d$, then there is a unique differentiation formula of 
exactness order $q$ given by $Df(\z)\approx\sum_{j=1}^N D\ell_j(\z)\, f(x_j)$,
where $\ell_j$, $j=1,\ldots,N$, are the Lagrange functions. In this case the growth function for 
$\dnorm=\dnorm_1$ coincides with the classical \emph{Lebesgue function},
$$
\rho_{q,D}(\z,\X,\dnorm_1)=\cL_{q,D}(\z,\X):=\sum_{j=1}^N |D\ell_j(\z)|.$$ 
If $\X$ is a unisolvent set for $\Pi_q^d$, then the monotonicity \eqref{gmon} of the growth function
implies that
$$
\rho_{q,D}(\z,\X,\dnorm_1)\le \min\big\{\cL_{q,D}(\z,\X'): \X'\subset\X,\;\text{$\X'$ is an interpolation set for $\Pi_q^d$} \big\}.$$

\begin{remark}\rm
  Growth functions
  in the form \eqref{gf} can be estimated
  from 
  above under appropriate assumptions on the set $\X$, such as sufficient density in a domain
satisfying an interior cone condition, see e.g.~\cite[Theorem~3.8]{wendland:2005-1}, or the existence of a
subset with a known bound on the Lebesgue function, for example Padua points \cite{CMV2005}. We will not pursue
such estimates in this paper.
\end{remark}

In what follows we will investigate weight vectors that realize the infimum in \eqref{dual} for
certain types of seminorms $\dnorm$.

\begin{definition}\label{dmin}\rm
Let $\dnorm$ be an seminorm on $\RR^N$.
A differentiation formula (\ref{ndif}) for a linear differential operator $D$ 
of order $k$ is said to be \emph{$\dnorm$-minimal of order $q$} if the weight vector 
$\w^*$ that defines it satisfies
$$
\|\w^*\|=\inf\Big\{\|\w\|:D p(\z)=\sum_{j=1}^Nw_jp(\x_j)
\text{ for all }p\in\Pi_{q}^d\Big\}.$$
Then $\|\w^*\|=\rho_{q,D}(\z,\X,\dnorm)$ by \eqref{dual}.
\end{definition}

From Propositions~\ref{Thepb}, \ref{Copb}, and equations \eqref{pebpw}, 
\eqref{dual} we immediately obtain the following error bounds.

\begin{theorem}\label{gft}
Assume that the differentiation formula (\ref{ndif})  with the weight vector $\w^*$
for a linear differential operator $D$ of order $k$ is $\dnorm$-minimal of order $q>k$ for an absolute 
seminorm $\dnorm$, and let $\Omega$ be a domain containing the set $S_{\z,\X}=\bigcup_{i=1}^N[\z,\x_i]$. Then 
 for any $r=k,\ldots,q-1$, and $\gamma\in(0,1]$, 
$$ 
|D f(\z)-\sum_{j=1}^{N} w_j^*\, f( \x_j )|\le
\rho_{q,D}(\z,\X,\dnorm)
\,\|\bdelta^{r+\gamma}(\z,\X)\|^*
\,|f|_{r,\gamma,\Omega},\quad f\in C^{r,\gamma}(\Omega),
$$ 
where the vector $\bdelta^{r+\gamma}(\z,\X)$ is defined in \eqref{delta}. If $\Omega$ contains the convex hull of $\{\z,\x_1,\ldots,\x_N\}$, then
$$ 
|D f(\z)-\sum_{j=1}^{N} w_j^*\, f( \x_j )|\le
\rho_{q,D}(\z,\X,\dnorm)
\,\|\bdelta^{r+1}(\z,\X)\|^*
\,|f|_{\infty,r+1,\Omega},\quad f\in W^{r+1}_\infty(\Omega).
$$ 
\end{theorem}

Note that similar estimates involving general growth functions $\rho_{q,D}(\z,\X,\dnorm)$ 
can be obtained for the error of the
kernel-based numerical differentiation, generalizing the results in \cite{DavySchaback16}. Indeed,  
\eqref{pebpw} and \eqref{dual} can be applied to the bound given in \cite[Lemma 7]{DavySchaback16}.

For the scaled formulas \eqref{ndifsc} for a homogeneous operator $D$ of order $k$
the error bound (\ref{peb}) takes the form
$$
|D f(\z)-\sum_{j=1}^{N} w^h_j\, f( \x^h_j )|\le
h^{q-1+\gamma-k}|f|_{q-1,\gamma,\Omega^h}
\sum_{j=1}^{N} |w_j|\|\x_j-\z\|_2^{q-1+\gamma},$$
where $\Omega^h$ contains the scaled set $S_{\z,\X^h}=\bigcup_{i=1}^N[\z,\x^h_i]$. 
Hence, for a $\dnorm$-minimal weight vector $\w$ we immediately obtain the following estimates
showing the correct scaling behavior of the error for any $h>0$,
\begin{align}\label{sce}
\begin{split}
|D f(\z)-\sum_{j=1}^{N} w^h_j\, f( \x^h_j )|&\le
h^{r+\gamma-k}\rho_{q,D}(\z,\X,\dnorm)
\,\|\bdelta^{r+\gamma}(\z,\X)\|^*
\,|f|_{r,\gamma,\Omega^h},%
\\
|D f(\z)-\sum_{j=1}^{N} w^h_j\, f( \x^h_j )|&\le
h^{r+1-k}\rho_{q,D}(\z,\X,\dnorm)
\,\|\bdelta^{r+1}(\z,\X)\|^*
\,|f|_{\infty,r+1,\Omega^h},%
\end{split}
\end{align}
where $r=k,\ldots,q-1$, $\{\z,\x_1,\ldots,\x_N\}\subset\Omega$, and 
for the second estimate we assume that $\Omega$ contains the convex hull of $\{\z,\x_1,\ldots,\x_N\}$.

\begin{remark}\label{scalable}\rm
If $\w$ is $\dnorm$-minimal for a homogeneous operator $D$ of order $k$ on $\z,\X$, then clearly $\w^h=h^{-k}\w$ is also $\dnorm_h$-minimal
 for the same operator $D$ on $\z,\X^h$ if $\|\w^h\|_h=\alpha_h\|\w\|$ with some $\alpha_h>0$. Therefore the $\dnorm$-minimal formulas are \emph{scalable} in the sense of \cite{DavySchabackOPT},
which is helpful for the computation of $\w$, see Section~\ref{compute}.
\end{remark}

\section{Weighted $\ell_1$-Minimal Differentiation Formulas}\label{ell1}

In the case of the $\ell_1$-norm $\dnorm= \dnorm_{1}$, with $\|\w\|_1:=\sum_{j=1}^{N} |w_j|$,
 and the (semi)norm $\dnorm= \dnorm_{1,\mu}$ 
of \eqref{wl1norm}
we will use the simplified notation for the growth function
\begin{align*}
\rho_{q,D}(\z,\X,1,\mu):=\rho_{q,D}(\z,\X,\dnorm_{1,\mu}),\quad \mu\ge0\quad
(\dnorm_{1,0}:= \dnorm_{1}).
\end{align*}
Thus,
\begin{align*}
\rho_{q,D}(\z,\X,1,\mu)&=\sup\{D p(\z):p\in \Pi_q^d,\;
|p(\x_j)|\leq \|\x_j-\z\|_2^{\mu},\;j=1,\ldots,N\}\\
&=\inf\Big\{\sum_{j=1}^{N} |w_j|\|\x_j-\z\|_2^{\mu}:\w\in\RR^N,\,D p(\z)=\sum_{j=1}^Nw_jp(\x_j)
\text{ for all }p\in\Pi_q^d\Big\},
\end{align*}
where $\|\x_j-\z\|_2^{0}:=1$ in the undetermined case $\x_j=\z$.

\bigskip

We now consider various aspects of $\dnorm_{1,\mu}$-minimal formulas in some detail.

\subsection{Sparsity}

An important feature of $\dnorm_{1,\mu}$-minimal formulas is a relatively small number of nonzero
weights $w_j$.

\begin{theorem}\label{Thesparse}
Given $D$, $\X$ and $\z$, assume that there exists a numerical differentiation
formula \eqref{ndif} of polynomial exactness order $q$. Then for any  
$\mu\ge0$, there is a $\dnorm_{1,\mu}$-minimal formula of order $q$ 
with at most $\dim\Pi^d_{q}$ nonzero weights.
\end{theorem} 
 
\begin{proof} 
  The constrained
  weighted $\ell_1$-minimization problem for the 
$\dnorm_{1,\mu}$-minimal weight vector
$\w=[w_1,\ldots,w_N]^T$ can be reformulated as 
a linear optimization problem
$$
\min_{\w^+,\w^-}\c^T(\w^++\w^-)  \quad\text{under}\quad \bA(\w^+ -\w^-)=\b,
$$
for nonnegative variables $\w^+,\w^-\in\mathbb{R}^N$, where $w^+_j=\max\{w_j,0\}$,
$w^-_j=-\min\{w_j,0\}$, $j=1\ldots,N$, such that $\w=\w^+ -\w^-$ and
$|w_j|=w^+_j -w^-_j$. Here $\c$ is  a fixed nonnegative vector with
$c_j=\|\x_j-\z\|_2^\mu$, and $\bA$ is a matrix with $\dim\Pi^d_{q}$ rows and
$N$ columns that expresses the polynomial exactness of order $q$.
This is a problem in standard normal form, and,
if solvability is assumed, there always is an optimal
vertex solution having no more positive components than the number of rows of
$\bA$. 
\end{proof}

Clearly, any linear programming algorithm that ends in a vertex solution,
for example the simplex method, will find a weight vector $\w$ satisfying 
Theorem~\ref{Thesparse}. 
Note that $\Pi^d_{q}$--unisolvent sets $\X$ will always consist of 
$N\geq \dim\Pi^d_{q}$ points, and if there are more points than needed, 
the $\dnorm_{1,\mu}$-minimal formulas do not use them, albeit the choice of the subset 
depends on $\mu$, see some numerical examples in Section~\ref{expmu}. 
Sparser solutions with less than $\dim\Pi^d_{q}$ nonzero weights $w_j$, i.e.\ the 
five point star for bivariate Laplace operator, are more difficult to find unless $\X$ is
specifically designed to admit such solutions, which is the standard approach in the finite
difference method on regular nodes.
Note that methods for obtaining sparse solution of underdetermined systems of
linear equations are the subject of Compressed Sensing \cite{FR13}.

\subsection{Error bounds and choice of $\mu$}
Given a linear differential operator $D$ of order $k$ and the order $q>k$ of polynomial exactness,
assume that $f\in C^{r,\gamma}(\Omega)$ for a domain $\Omega$ containing $S_{\z,\X}$, with some $r\in\{k,\ldots,q-1\}$
and $\gamma\in(0,1]$. We consider $\dnorm_{1,\mu}$-minimal formulas for various $\mu\ge0$.

The choice $\mu=r+\gamma$ plays a special role as it delivers the best possible estimate
in Theorem~\ref{gft}. Indeed, since 
$\|\bdelta^{r+\gamma}(\z,\X)\|_{1,r+\gamma}^*=1$ by \eqref{dwl1norm}, we have
\begin{equation}\label{peb1}
|D f(\z)-\sum_{j=1}^{N} w_j^{1,r+\gamma}\, f( \x_j )|\le
\rho_{q,D}(\z,\X,1,r+\gamma)   
\,|f|_{r,\gamma,\Omega} .
\end{equation}
where $\w^{1,r+\gamma}$ is the weight vector of a  $\dnorm_{1,r+\gamma}$-minimal formula.
In view of \eqref{pebpw} and  \eqref{dual},  
$$ %
\rho_{q,D}(\z,\X,1,r+\gamma)\le\rho_{q,D}(\z,\X,\dnorm)
\,\|\bdelta^{r+\gamma}(\z,\X)\|^*
$$ %
for any absolute seminorm $\dnorm$, which shows that \eqref{peb1} is the best 
estimate obtainable from Theorem~\ref{gft}.

From \eqref{peb1} we obtain in particular that for any $f\in W^q_\infty(\Omega)$, 
where $\Omega$ contains the convex hull of $S_{\z,\X}$, any 
$\dnorm_{1,q}$-minimal formula of exactness order $q$ satisfies 
\begin{equation}\label{peb1a}
|D f(\z)-\sum_{j=1}^{N} w_j^{1,q}f( \x_j )|\le
\rho_{q,D}(\z,\X,1,q)
\,|f|_{\infty,q,\Omega}.
\end{equation}

Nevertheless, other choices of $\mu$ are also of interest. In particular, a $\dnorm_{1,0}$-minimal formula has 
the \emph{optimal stability constant} 
$$
\|\w^{1,0}\|_1=\rho_{q,D}(\z,\X,1,0)$$
over all  weight vectors of the fixed exactness order $q>k$, whereas a larger $\mu>0$ means that the weights in \eqref{wl1norm}
penalize more distant points in $\X$, thus leading to smaller \emph{effective supports}
$$
\X_\w:=\{x_j\in\X:w_j\ne0\}$$
of the respective formulas.

Therefore we derive error bounds for $\dnorm_{1,\mu}$-minimal formulas with any $\mu\ge0$, that also take into account the influence of the effective
supports $\X_\w$.

\begin{theorem}\label{pebt}
Assume that the differentiation formula (\ref{ndif})  with the weight vector $\w=\w^{1,\mu}$
for a linear differential operator $D$ of order $k$ is $\dnorm_{1,\mu}$-minimal of order $q>k$, 
and let $\Omega$ be a domain containing the set $S_{\z,\X}=\bigcup_{i=1}^N[\z,\x_i]$. Then 
 for any $r=k,\ldots,q-1$, and $\gamma\in(0,1]$, 
\begin{align}\label{peb2}
|D f(\z)-\sum_{j=1}^{N} w_j^{1,\mu}\, f( \x_j )|&\le
\rho_{q,D}(\z,\X,1,\mu)\,  h_{\z,\X_\w}^{r+\gamma-\mu} 
\,|f|_{r,\gamma,\Omega}\quad\text{if}\quad 0\le \mu\le r+\gamma,\\
\label{peb3}
|D f(\z)-\sum_{j=1}^{N} w_j^{1,\mu}\, f( \x_j )|&\le
\rho_{q,D}(\z,\X,1,\mu)\,  s_{\z,\X_\w}^{r+\gamma-\mu} 
\,|f|_{r,\gamma,\Omega} \quad\text{if}\quad \mu> r+\gamma, 
\end{align}
where $h_{\z,\X}$ is defined in \eqref{hzXdef}, and
\begin{equation}\label{szX}
s_{\z,\Y}:=\min_{\y\in\Y\setminus\{\z\}}\|\y-\z\|_2 %
\end{equation}
denotes the distance from $\z$ to $\Y\setminus\{\z\}$. 
\end{theorem}

\begin{proof}
The case $\mu=r+\gamma$ is covered in \eqref{peb1}. 
Let $\mu\ne r+\gamma$. In view of \eqref{dwl1norm},
$$ %
\|\bdelta^{r+\gamma}(\z,\X_\w)\|_{1,\mu}^*=\max_{\x_j\in\X_\w\setminus \{\z\}} \|\x_j-\z\|_2^{r+\gamma-\mu}
=\begin{cases}
h_{\z,\X_\w}^{r+\gamma-\mu},& \hbox{if }\mu<r+\gamma,\\
s_{\z,\X_\w}^{r+\gamma-\mu},& \hbox{if }\mu>r+\gamma.
\end{cases}
$$ %
Moreover, it is clear from \eqref{dual} that %
$\rho_{q,D}(\z,\X,1,\mu)=\rho_{q,D}(\z,\X_\w,1,\mu)$,
and both \eqref{peb2} and \eqref{peb3} follow from Theorem~\ref{gft} applied to the set $\X_\w$ instead of $\X$.
\end{proof}

Note that $h_{\z,\X_\w}\le h_{\z,\X}$ and $s_{\z,\X_\w}^{-1}\le s_{\z,\X}^{-1}$, and 
the respective quantities may be significantly different  if $\w$ is sparse.

When comparing the bounds in \eqref{peb1}, \eqref{peb2}, \eqref{peb3} it is helpful to take
into account the scaling behavior of the growth functions $\rho_{q,D}(\z,\X,1,\mu)$. If 
$D$ is  a homogeneous operator of order $k$, then for $\X^h$ defined by \eqref{Xhwh},
\begin{equation}\label{rhos}
\rho_{q,D}(\z,\X^h,1,\mu)=h^{\mu-k}\rho_{q,D}(\z,\X,1,\mu),\quad \mu\ge0,
\end{equation}
which is easy to check by the definition of $\rho_{q,D}(\z,\X,1,\mu)$. In particular, all these bounds give
convergence with rate $\mathcal{O}(h^{r+\gamma-k})$ as in \eqref{sce} for the scaled formulas with 
weights $w_j^h=h^{-k}w_j$.

By  rearranging the factors in \eqref{peb2} and estimating $h_{\z,\X_\w}$ by $h_{\z,\X}$ we may
 obtain bounds in the form \eqref{consist},
\begin{equation}\label{peb2s}
|D f(\z)-\sum_{j=1}^{N} w_j^{1,\mu}\, f( \x_j )|\le
\sigma_{q,D}(\z,\X,1,\mu)\,  h_{\z,\X}^{r+\gamma-k} 
\,|f|_{r,\gamma,\Omega},\quad \mu\le r+\gamma,
\end{equation}
where the $\sigma$-factor
$$ 
\sigma_{q,D}(\z,\X,1,\mu):=h_{\z,\X}^{k-\mu}\rho_{q,D}(\z,\X,1,\mu)
$$ 
 is scale-invariant in the sense that
$\sigma_{q,D}(\z,\X,1,\mu)=\sigma_{q,D}(\z,\X^h,1,\mu)$ for any $h>0$
and any  homogeneous $D$. It is connected to  $\sigma(\z,\X,\w,\mu)$ of \eqref{sig} 
by the relation
$$
\sigma_{q,D}(\z,\X,1,\mu)=
\inf\Big\{\sigma(\z,\X,\w,\mu):\w\in\RR^N,\,D p(\z)=\sum_{j=1}^Nw_jp(\x_j)\Big\}.$$

Although the right hand side of \eqref{peb2s} is minimal for $\mu=r+\gamma$ thanks to \eqref{sigs}, the
more accurate estimates in \eqref{peb2} are not so conclusive. In Section~\ref{expmu} we investigate the
influence of $\mu$ on the accuracy  and stability of the $\dnorm_{1,\mu}$-minimal formulas numerically. The
results suggest that significantly higher accuracy can be achieved at the expense of a moderate increase
of  $\|\w\|_1$ when using $\mu>0$. However, an excessively large $\mu$ may lead to a high error and a high stability constant
$\|\w\|_1$ on more difficult sets $\X$. The choice $\mu=r+\gamma$ seems to
deliver
a good compromise between stability and accuracy.
The error bounds \eqref{peb3} involving $s_{\z,\X_\w}$ rather than $h_{\z,\X_\w}$ are much less accurate 
than those in \eqref{peb1}--\eqref{peb2} in these experiments, when compared with the actual error of the numerical differentiation formulas.

\subsection{Positive formulas} \label{positive}

For an elliptic differential operator $D$ of second order, a numerical differentiation formula \eqref{ndif} with $\x_1=\z$ is said to be \emph{positive}
if $w_1<0$ and $w_j>0$, $j=2,\ldots,N$. Formulas of this type are useful in generalized finite difference methods \cite{Seibold08} for elliptic 
PDEs since under certain additional assumptions the system matrices become M-matrices with highly desirable properties such as
guaranteed invertibility and discrete maximum principle.

\begin{proposition}\label{Lemnonex}
There are no positive formulas  that are exact for polynomials of order 5 or higher.
\end{proposition} 
\begin{proof}{}
Following the argumentation in \cite[p.~57]{SeiboldPhD06}, assume that  a positive formula
\eqref{ndif} is exact for polynomials of order 5. Then in particular it is exact for the
polynomial $p(\x)=\|\x-\z\|_2^4\in\Pi^d_5$. Since $D p(\z)=0$, this implies that 
$\sum_{j=1}^Nw_j\|\x_j-\z\|_2^4=0$ and hence $w_1=\cdots=w_N=0$, a contradiction.
\end{proof}

An example of a positive formula of polynomial exactness order 4 for the Laplace operator $\Delta$ is
the classical 5 point star formula in $\RR^2$.

The following theorem shows a minimality property and an 
error bound for arbitrary positive formulas.

\begin{theorem}\label{tpse}
Let $D$ be an elliptic differential operator \eqref{Dop} of second order, and let a numerical differentiation 
formula (\ref{ndif}) with weight vector $\w$ be positive and exact for polynomials of order $q\in \{3,4\}$. Then $\w$ is 
$\dnorm_{1,2}$-minimal, with 
\begin{equation}\label{pse0}
\|\w\|_{1,2}=\rho_{q,D}(\z,\X,1,2)=\tau_D(\z):=2\sum_{|\alpha|=1}c_{2\alpha}(\z),
\end{equation}
and hence for any $f\in C^{r,\gamma}(\Omega)$, with  $2\le r+\gamma\le q$,
\begin{equation}\label{pse}
|D f(\z)-\sum_{j=1}^{N} w_j\, f( \x_j )|\le \tau_D(\z)\, h_{\z,\X_\w}^{r+\gamma-2}|f|_{r,\gamma,\Omega}
\end{equation}
where $\Omega\subset\RR^d$ is any domain that contains the set $S_{\z,\X}$ defined by \eqref{SzX}, 
and $\X_\w:=\{x_j\in\X:w_j\ne0\}$.  
\end{theorem}

\begin{proof}
If a formula (\ref{ndif}) with weight vector $\w$ is exact for polynomials of order $q$, then in particular it is exact for 
$p(\x)=\|\x-\z\|_2^2\in\Pi^d_3$, and hence
 $$
\sum_{j=1}^Nw_j\|\x_j-\z\|_2^2=Dp(\z)=\tau_D(\z).$$
It follows that $\|\w\|_{1,2}\ge \tau_D(\z)$. (Note that $\tau_D(\z)>0$ for any elliptic operator $D$ of second order.)
If we assume that $\x_1=\z$ and $\w$ is positive, then $\|\w\|_{1,2}=\sum_{j=2}^Nw_j\|\x_j-\z\|_2^2=\tau_D(\z)$, which proves 
 the $\dnorm_{1,2}$-minimality of the positive formula, and \eqref{pse0} in view of \eqref{dual}.
The error bound \eqref{pse} follows from \eqref{peb2}.
\end{proof}

As a corollary we obtain the following statement about growth functions.

\begin{corollary}\label{cpse}
Let $D$ be an elliptic differential operator \eqref{Dop} of second order, and let
$\z$ and $\X$ be such that there exists a positive differentiation formula (\ref{ndif})
of exactness order $q\in \{3,4\}$. Then $\rho_{q,\Delta}(\z,\X,1,2)=\tau_D(\z)$.
\end{corollary}

Note that for the Laplace operator $D=\Delta$ we have $\tau_\Delta(\z)=2d$ for all $\z\in\RR^d$.

Theorem~\ref{tpse}  shows that any positive differentiation formula is $\dnorm_{1,2}$-minimal.
If there is more than one such formula, then all of them have the same $\dnorm_{1,2}$-seminorm $\|\w\|_{1,2}=\tau_D(\z)$ 
and satisfy the error bound \eqref{pse}. It is suggested in \cite{Seibold08} to choose a particular positive differentiation formula
by minimizing the $\dnorm_{1,\mu}$-seminorm of the weight vector $\w$, with $\mu>2$, which is supported by numerical evidence. Indeed,
for a greater $\mu$, points closer to $\z$ are preferred, which potentially improves the bound in \eqref{pse} thanks to a smaller $h_{\z,\X_\w}$.
Our experiments in Section~\ref{expmu} also investigate this phenomenon, although we do not require the formulas to be positive.

\section{Numerical Differentiation by Least Squares} \label{ell2}

General results of Section~\ref{gfun} can also be applied to the numerical differentiation formulas obtained by differentiating 
the least squares polynomial fits to the data.

Assuming that $\X=\{\x_1,\ldots,\x_N\}$, where $N\ge\dim \Pi^d_{q}$,
is a unisolvent set for $\Pi^d_{q}$, the weighted least squares polynomial 
$L^\btheta_{\X,q}f\in\Pi^d_{q}$, $\btheta=[\theta_1,\ldots,\theta_N]^T$,
$\theta_j>0$, is uniquely defined by the condition
$$
\|(L^\btheta_{\X,q}f-f)|_\X\|_{2,\btheta}=\min\{\|(p-f)|_\X\|_{2,\btheta}:p\in \Pi^d_{q}\},$$
where
$$
\|\bv\|_{2,\btheta}:=\Big(\sum_{j=1}^N\theta_jv_j^2\Big)^{1/2}.$$
Moreover, $L^\btheta_{\X,q}f$ satisfies $L^\btheta_{\X,q}p=p$ for all $p\in\Pi^d_{q}$.  Since $L^\btheta_{\X,q}f$ depends linearly
on $f|_\X$, the application of $D$ to $L^\btheta_{\X,q}f$ leads to a numerical differentiation formula 
\begin{equation}\label{ndiflw} 
Df(\z)\approx D L^\btheta_{\X,q}f(\z)=\sum_{j=1}^{N} w^{2,\btheta}_j f( \x_j )
\end{equation}
of exactness order $q$. Note that such formulas are frequently used in the generalized finite difference methods, 
see e.g.~\cite{LiszOrk80,Benito03}.

Let $\dnorm_{2,\btheta^{-1}}$ be the weighted $\ell_2$-norm defined by the weight vector
$\btheta^{-1}=[\theta_1^{-1},\ldots,\theta_N^{-1}]^T$,
$$
\|\bv\|_{2,\btheta^{-1}}=\Big(\sum_{j=1}^N\frac{v_j^2}{\theta_j}\Big)^{1/2}.$$
It is well known  %
(see e.g.\ \cite[Section 22.3]{fasshauer:2007-1}) %
that the weight vector $\w^{2,\btheta}$ of \eqref{ndiflw}  solves the quadratic minimization problem
$$
\|\w\|_{2,\btheta^{-1}}^2=\sum_{j=1}^N\frac{w_j^2}{\theta_j}\to\min\quad\text{subject to }
Dp(\z)= \sum_{j=1}^Nw_jp(\x_j)\text{ for all }p\in\Pi^d_q.$$
Hence, \eqref{ndiflw} is a $\dnorm_{2,\btheta^{-1}}$-minimal formula of order $q$ according to Definition~\ref{dmin}.
We may also allow zero and infinite weights, $0\le\theta_j\le\infty$.
If $\theta_{j_0}=0$ for some $j_0$, then $w^{2,\btheta}_{j_0}=0$. If $\theta_{j_0}=\infty$  then 
the weighted least squares polynomial
$L^\btheta_{\X,f}$ satisfies the interpolation condition $L^\btheta_{\X,f}(\x_{j_0})=f(\x_{j_0})$.
In both cases $\dnorm_{2,\btheta}$ and  $\dnorm_{2,\btheta^{-1}}$ lose the $j_0$-th term and become seminorms.

The growth function corresponding to the seminorm $\dnorm_{2,\btheta^{-1}}$ is
$$ %
\rho_{q,D}(\z,\X,\dnorm_{2,\btheta^{-1}})=\|\w^{2,\btheta}\|_{2,\btheta^{-1}}
=\sup\Big\{Dp(\z):p\in \Pi^d_{q},\;
\sum_{j=1}^N\theta_j|p(\x_j)|^2\le1\Big\}.
$$ %
Since $\dnorm_{2,\btheta^{-1}}$ is dual to $\dnorm_{2,\btheta}$, 
Theorem~\ref{gft} implies the following error bounds.

\begin{theorem}\label{lsgft}
Let  $D$ be a linear differential operator  of order $k$ and let $\Omega$ be a domain containing 
the set $S_{\z,\X}=\bigcup_{i=1}^N[\z,\x_i]$. The numerical differentiation formula (\ref{ndiflw}) of order $q>k$
 for any $r=k,\ldots,q-1$, and $\gamma\in(0,1]$ satisfies 
$$ 
|D f(\z)-\sum_{j=1}^{N} w_j^{2,\btheta}\, f( \x_j )|\le
\rho_{q,D}(\z,\X,\dnorm_{2,\btheta^{-1}})
\Big(\sum_{j=1}^{N} \theta_j\|\x_j-\z\|_2^{2(r+\gamma)}\Big)^{1/2}
\,|f|_{r,\gamma,\Omega},
$$ 
for all $f\in C^{r,\gamma}(\Omega)$.
If $\Omega$ contains the convex hull of $\{\z,\x_1,\ldots,\x_N\}$, then for all $f\in W^{r+1}_\infty(\Omega)$,
$$ 
|D f(\z)-\sum_{j=1}^{N} w_j^{2,\btheta}\, f( \x_j )|\le
\rho_{q,D}(\z,\X,\dnorm_{2,\btheta^{-1}})
\Big(\sum_{j=1}^{N} \theta_j\|\x_j-\z\|_2^{2r+2}\Big)^{1/2}
\,|f|_{\infty,r+1,\Omega}.
$$ 

\end{theorem}

Motivated by the $\dnorm_{1,\mu}$-minimal formulas studied in
Section~\ref{ell1},
we consider the least squares solution
with the weights given by
$$ 
\theta_j=\|\x_j-\z\|_2^{-2\mu}, \quad j=1,\ldots,N,\quad \mu\ge0,
$$ 
and denote by $\w^{2,\mu}$ the vector $\w^{2,\btheta}$ obtained with these weights, which are therefore
minimal with respect to the seminorm 
$$
\|\w\|_{2,\mu}:=\Big(\sum_{j=1}^{N} w_j^2\|\x_j-\z\|_2^{2\mu}\Big)^{1/2},\quad \mu>0,\qquad \|\w\|_{2,0}:=\|\w\|_2.$$
The corresponding growth functions $\rho_{q,D}(\z,\X,2,\mu):=\rho_{q,D}(\z,\X,\dnorm_{2,\mu})$ satisfy
\begin{align}\label{2mudual}
\begin{split}
\rho_{q,D}(\z,\X,2,\mu)
&=\sup\Big\{Dp(\z):p\in \Pi^d_{q},\;
\sum_{j=1\atop \x_j\ne\z}^N\frac{|p(\x_j)|^2}{\|\x_j-\z\|_2^{2\mu}}\le1\Big\}.\\
&=\|\w^{2,\mu}\|_{2,\mu}=\Big(\sum_{j=1}^{N} (w_j^{2,\mu})^2\|\x_j-\z\|_2^{2\mu}\Big)^{1/2}.\end{split}
\end{align}
In this case Theorem~\ref{lsgft} leads to the following statement.

\begin{corollary}\label{pebtls}
Given a linear differential operator $D$ of order $k$, the  $\dnorm_{2,\mu}$-minimal formula  of order $q>k$
with the weight vector $\w^{2,\mu}$, $\mu\ge0$, satisfies for any domain $\Omega$ containing the set $S_{\z,\X}=\bigcup_{i=1}^N[\z,\x_i]$
and any $r=k,\ldots,q-1$, and $\gamma\in(0,1]$, the error bound
\begin{equation}\label{lsb}
|D f(\z)-\sum_{j=1}^{N} w_j^{2,\mu}\, f( \x_j )|\le
\Big(\sum_{j=1\atop \x_j\ne\z}^{N} \|\x_j-\z\|_2^{2(r+\gamma-\mu)}\Big)^{1/2}
\rho_{q,D}(\z,\X,2,\mu)
\,|f|_{r,\gamma,\Omega}.
\end{equation}
In particular, for $\mu=r+\gamma$ we have
\begin{equation}\label{lsbwo}
|D f(\z)-\sum_{j=1}^{N} w_j^{2,r+\gamma}f( \x_j )|\le
\sqrt{N}\,
\rho_{q,D}(\z,\X,2,r+\gamma)
\,|f|_{r,\gamma,\Omega}.
\end{equation}
\end{corollary}

We can also estimate the error of the $\dnorm_{2,\mu}$-minimal least squares formulas
with the help of the $\dnorm_{1,\mu}$ growth function.
Indeed, from the inequality $\|\w\|_{2,\mu}\le\|\w\|_{1,\mu}\le\sqrt{N}\|\w\|_{2,\mu}$
it follows that
\begin{equation}\label{rho12}
\rho_{q,D}(\z,\X,2,\mu)\le\rho_{q,D}(\z,\X,1,\mu)\le\sqrt{N}\rho_{q,D}(\z,\X,2,\mu).
\end{equation}
Hence, \eqref{lsbwo} implies a bound for the error of the $\dnorm_{2,r+\gamma}$-minimal 
formulas that is only by a factor of $\sqrt{N}$ worse than the estimate \eqref{peb1} for the 
$\dnorm_{1,r+\gamma}$-minimal formulas.

Moreover, Theorem~\ref{pebt} implies in view of \eqref{rho12} that the error of 
the $\dnorm_{1,\mu}$-minimal formulas of exactness order $q$ can be estimated with the help of the 
$\dnorm_{2,\mu}$ growth function,
\begin{equation}\label{peb123s1}
|D f(\z)-\sum_{j=1}^{N} w_j^{1,\mu}\, f( \x_j )|\le
\sqrt{N}\rho_{q,D}(\z,\X,2,\mu)\,\max\big\{  h_{\z,\X}^{r+\gamma-\mu},  s_{\z,\X}^{r+\gamma-\mu}\big\}
\,|f|_{r,\gamma,\Omega},
\end{equation}
where $h_{\z,\X}$ is defined in \eqref{hzXdef}, and $s_{\z,\X}$ in \eqref{szX}.

In particular, for any $f\in W^q_\infty(\Omega)$, where $\Omega$ contains the convex hull of 
$\{\z,\x_1,\ldots,\x_N\}$,  \eqref{lsbwo} and \eqref{peb123s1} imply that both $\dnorm_{1,q}$-minimal
and $\dnorm_{2,q}$-minimal formulas of exactness order $q$ satisfy the same error bound
\begin{equation}\label{peb1s1}
|D f(\z)-\sum_{j=1}^{N} w_j^{1,q}f( \x_j )|\le
\sqrt{N}\,\rho_{q,D}(\z,\X,2,q)
\,|f|_{\infty,q,\Omega},
\end{equation}
\begin{equation}\label{lsbwoa}
|D f(\z)-\sum_{j=1}^{N} w_j^{2,q}f( \x_j )|\le
\sqrt{N}\,\rho_{q,D}(\z,\X,2,q)
\,|f|_{\infty,q,\Omega}.
\end{equation}

Numerical experiments in Section~\ref{numexp} suggest that the accuracy of the $\dnorm_{2,\mu}$-minimal  formulas
is close to that of the $\dnorm_{1,\mu}$-minimal formulas. Their weight vectors $\w$ can be more efficiently computed,
but they are not sparse.

\begin{remark}\label{rho2}\rm 
A remarkable feature of \eqref{peb123s1} and \eqref{peb1s1} is that these error bounds do not rely on the knowledge of the
$\dnorm_{1,\mu}$-minimal formula, so that its quality can be assessed by computing $\rho_{q,D}(\z,\X,2,\mu)$ and other
ingredients of \eqref{peb123s1}--\eqref{peb1s1} without resorting to \red{the relatively} expensive $\ell_1$-minimization. 
The latter will only be needed to compute  a sparse $\dnorm_{1,\mu}$-minimal formula 
after good values for $q$ and $\mu$ have been found based on the estimates. Such sparse formulas are
of interest for the generalized finite difference methods 
 since
they lead to sparser system matrices. Since the error of kernel-based numerical differentiation 
formulas is also bounded in terms of $\rho_{q,D}(\z,\X,1,q)$ \cite[Theorem 9]{DavySchaback16},
the inequalities \eqref{rho12} imply a bound in terms of the more efficiently computable $\rho_{q,D}(\z,\X,2,q)$ for these formulas as well,
which can be used in the kernel-based generalized finite difference methods to improve the algorithms for the selection of
local point sets that generate numerical differentiation formulas \cite{DavyOanh11,DavyOanh11sp,ODP17}.
\end{remark}

\begin{remark}\rm
Note that the stability constant $\|\w^{2,\theta}\|_1$ of the formula
\eqref{ndiflw} has an alternative interpretation as the 
\emph{Lebesgue function} $\cL^\theta_{q,D}(\z,\X)$ of the differentiated least squares operator $DL^\theta_{\X,q}$ because 
\begin{equation}\label{lebfls}
\cL^\theta_{q,D}(\z,\X):=\sup%
\{|DL^\theta_{\X,q}f(\z)|:%
|f(\x_i)|\le 1\}=\|\w^{2,\theta}\|_1.
\end{equation}
Similar to certain estimates of the \emph{Lebesgue constant} 
$\cL^\theta_{q,D}(\Omega,\X):=\sup_{\z\in\Omega}\cL^\theta_{q,D}(\z,\X)$ proposed in \cite{Davy02} the formula
\eqref{lebfls} can be employed to determine suitable degrees of local polynomial approximations in two-stage scattered
data fitting algorithms  \cite{DPR2014,DavyZeil04}.
\end{remark}

\red{
\begin{remark}\rm
Although $\dnorm_{2,\mu}$-minimal formulas are not sparse, they may be preferred over 
$\dnorm_{1,\mu}$-minimal formulas because they are cheaper to compute. Indeed, their computation requires solving
just one positive definite linear system 
if using the method of normal
equations, see Section 6.1, whereas $\ell_1$-minimization is done by iterative
algorithms such as Simplex Method. (Note however that \cite{Seibold10} provides numerical evidence of competitive cost performance of
Simplex Method for low order formulas with $q=3$.) Another advantage of the least squares formulas, 
relevant for the numerical stability and certain applications, is the continuous dependence of the weight vector $\w$ on the point positions $\X$,
which does not hold for the $\ell_1$--minimal formulas as observed in Remarks 1 and 2 of \cite{Seibold08}.
\end{remark}
}

\section{Numerical Experiments} \label{numexp}

\subsection{Computation of weight vectors}\label{compute}

Since $\Pi^d_q$ is shift-invariant, the polynomial exactness condition \eqref{qexact} for a vector $\w$ is equivalent to
$$
D_zp(\oo)=\sum_{j=1}^N w_jp(\x_j-\z) \quad\hbox{for all}\quad  p\in \Pi^d_{q},$$
where $\oo$ is the origin in $\RR^d$, and for any operator $D$ given by \eqref{Dop} the shifted operator $D_\z$ is defined by
$$
D_\z f(\x):=\sum_{|\alpha|\le k}c_\alpha(\x+\z)\partial^\alpha f(\x).$$
Hence we can always use $\z=\oo$ in the implementation of a $\dnorm$-minimal formula if we replace $\X$ by
$\X-\z$ and $D$ by $D_\z$.  This allows the use of
simple monomials for a basis of polynomials. 

Thus, for a $\dnorm_{1,\mu}$-minimal formula
we arrive at the following linear programming problem:
find $\w^{1,\mu}\in\RR^N$ that minimizes
$$
\sum_{j=1}^N|w_j|\|\tilde\x_j\|_2^\mu,\qquad \tilde \x_j:=\x_j-\z,
$$
subject to  the constraints
$$
\sum_{j=1}^Nw_j \tilde\x_j^\alpha =D_\z\x^\alpha|_{\x=\oo}=\alpha!\,c_\alpha(\z) ,\qquad |\alpha|< q.$$

However, for the sets $\X$ with small $h_{\z,\X}$ the matrix $[\tilde\x_j^\alpha]_{j,\alpha}$ is extremely ill-conditioned.
Therefore, we make use of the scalability of the $\dnorm$-minimal formulas, see Remark~\ref{scalable}.
Thanks to this property, for a \emph{homogenous} operator $D$ of order $k$, the vector $\w$ can be obtained 
by scaling $\w=h^{-k}_{\z,\X}\bv$ from the weight
vector $\bv\in\RR^N$ that solves the following problem:
\begin{equation}\label{wl1min}
\sum_{j=1}^N|v_j|\|\y_j\|_2^\mu\to\min%
\quad\text{subject to}\quad\frac{1}{\alpha!}\sum_{j=1}^Nv_j \y_j^\alpha =
c_\alpha(\z),\quad |\alpha|< q,
\end{equation}
where $\y_j:=h_{\z,\X}^{-1}(\x_j-\z)$, $j=1,\ldots,N$.

We refer the reader to the extensive literature on the algorithms for the \emph{basis pursuit}, the name often used for
the $\ell_1$ minimization problem, of which \eqref{wl1min} is a special case. 
Specifically in the area of Compressed Sensing  there is a high demand for such algorithms delivering sparse solutions.
In our experiments we obtain sparse weights $\bv$ with 
 $v_j=0$, $j\notin I$, for some $I\subset\{1,\ldots,N\}$ by MATLAB  command
{\tt linprog} by specifying `dual-simplex' as algorithm.
Since MATLAB's Optimization Toolbox currently only works in double precision, we recompute the weights in  
the variable-precision arithmetic by solving the constraint equations for 
$v_j$, $j\in I$, should a higher accuracy be needed.

Since $\dnorm_{2,\mu}$-minimal formulas are also scalable, the computation of their weight vectors $\w$ is more stable numerically if we rescale the set $\X$ 
into the unit disk as $\Y:=h_{\z,\X}^{-1}(\X-\z)$, obtain the weight vector $\bv$ of the formula
$$
D_\z L^\btheta_{\Y,q}f(\oo)=\sum_{j=1}^{N} v_j f( \y_j ),\quad\text{with}\quad \theta_j=\|\y_j\|_2^{-2\mu},\quad j=1,\ldots,N,$$
and then scale back to arrive at $\w=h^{-k}_{\z,\X}\bv$. To see how  $\bv$ can be computed, write the polynomial $p=L^\btheta_{\Y,q}f\in\Pi^d_q$
in the monomial basis,
$$
p(\y)=\sum_{|\alpha|<q}b_\alpha\y^\alpha,\quad b_\alpha\in\RR.$$

Assuming that $\oo\notin\Y$, the vector $\b=[b_\alpha]_\alpha$ is the unique  solution of the problem
$$
\|W\b-f_\Y\|_{2,\btheta}\to\min,\qquad W:=[\y_j^\alpha]_{j,\alpha},\quad f_\Y:=[f(\y_1),\ldots,f(\y_N)]^T.$$
(Note that $W$ has full rank since $\Y$ is unisolvent for $\Pi^d_q$.)
Hence 
$$
\b=(\delta W)^+\delta f_\Y,$$
where $\delta:=\diag(\sqrt{\theta_1},\ldots,\sqrt{\theta_N})$ and $A^+$ denotes the Moore-Penrose pseudoinverse of a matrix $A$.
Since
$$
D_\z p(\oo)=\sum_{|\alpha|<q}b_\alpha D_\z\y^\alpha|_{\y=\oo}=[\alpha!\,c_\alpha(\z)]_\alpha^T\, \b,$$
we arrive at the formula
$$ 
\bv^T=[\alpha!\,c_\alpha(\z)]_\alpha^T(\delta W)^+\delta. 
$$ 

In the case $\oo\in\Y$ assume without loss of generality that $\y_1=\oo$. Then $\theta_1=\infty$ and
hence $b_\oo=f(\y_1)$ and $\tilde\b:=[b_\alpha]_{\alpha\ne\oo}$ is the unique  solution of the least squares problem
$$
\|\tilde W\tilde\b-\tilde f_\Y\|_{2,\btheta}\to\min,$$
where
$$
\tilde W:=[\y_j^\alpha]_{j\ne 1,\alpha\ne\oo},\quad \tilde f_\Y:=[f(\y_2)-f(\y_1),\ldots,f(\y_N)-f(\y_1)]^T,$$
with $\tilde W$ necessarily a full rank matrix. Then 
\begin{align*}
D_\z p(\oo)&=c_\oo(\z) f(\y_1)+[\alpha!\,c_\alpha(\z)]_{\alpha\ne\oo}^T\, \tilde\b\\
&=c_\oo(\oo) f(\y_1)+[\alpha!\,c_\alpha(\z)]_{\alpha\ne\oo}^T\,(\tilde\delta\tilde W)^+\tilde\delta \tilde f_\Y,
\end{align*}
with $\tilde\delta:=\diag(\sqrt{\theta_2},\ldots,\sqrt{\theta_N})$, which shows that 
\begin{align*} 
[v_2,\ldots,v_N]&=[\alpha!\,c_\alpha(\z)]_{\alpha\ne\oo}^T\,(\tilde\delta\tilde W)^+\tilde\delta, \\
v_1&=c_\oo(\z)-\sum_{j=2}^N v_j.
\end{align*}

The pseudoinverse of $\delta W$ or  $\tilde\delta\tilde W$ can be found by using either the singular value decomposition, the
QR-factorization or the normal equations \cite{Stewart98}. The last method leads to a simple formula 
$A^+=(A^T\!A)^{-1}A^T$, but requires 
solving \red{a linear system with the matrix $A^T\!A$ whose condition number is the square of the condition number
of $A$, which is undesirable if $A$ is already ill-conditioned. We} 
are using the singular value decomposition in the
numerical experiments as the most reliable method.

Once the weight vector $\w$ of a $\dnorm_{2,\mu}$-minimal formula of order $q$ has been computed, the
growth function $\rho_{q,D}(\z,\X,2,\mu)$ can be easily evaluated by \eqref{2mudual}.

\subsection{Comparison of numerical differentiation formulas}
To compare various numerical differentiation formulas we follow the same approach as in \cite{DavySchaback16}, 
and use  the worst case error on Sobolev spaces $H^\rho(\RR^d)=W^\rho_2(\RR^d)$ as the main accuracy measure. 
By Sobolev theorem, under certain restrictions on $\Omega$,
$H^\rho(\Omega)$ is embedded in $C^{r,\gamma}(\Omega)$ if $\rho-d/2=r+\gamma$ with $\rho\in\ZZ_+$ 
and $\gamma\in(0,1)$. Therefore the estimates obtained in this paper are applicable to
$f\in H^\rho(\Omega)$ if  $\rho>d/2$ and $\rho-d/2\notin\NN$. The reason to deal with the spaces $H^\rho(\Omega)$ 
rather than $C^{r,\gamma}(\Omega)$ is that the
worst case error of the formula \eqref{ndif} as well as the optimal recovery error are easily computable for them.

Recall that  the space $H^\rho(\RR^d)$ with the norm 
$$
\|f\|_{H^\rho(\RR^d)}
:=(2\pi)^{-d/4}\Big(\int_{\RR^d}|\hat f(\bomega)|^2(1+\|\bomega\|_2^2)^\rho\,\d\bomega\Big)^{1/2}
$$
in the case $\rho>d/2$ coincides with the \emph{native space} (see e.g.\ \cite{wendland:2005-1}) of the \emph{Mat\'ern kernel}
$$
M_{\rho,d}(\x):=\frac{\mathcal{K}_{\rho-d/2}(\|\x\|_2)\|\x\|_2^{\rho-d/2}}{2^{\rho-1}\Gamma(\rho)},\quad \rho>d/2,$$
where $\mathcal{K}_{\nu}$ denotes the modified Bessel function of second kind.
It follows (see \cite{DavySchaback16}) that the \emph{worst case error} of a numerical differentiation formula \eqref{ndif} on 
the unit ball of $H^\rho(\RR^d)$ can be computed by
\begin{equation}\label{optest}
\sup_{\|f\|_{H^\rho(\RR^d)}\le 1}|D f(\z)-\sum_{j=1}^Nw_jf(x_j)|=\sqrt{Q_{D,\X}(\w)},
\end{equation}
where %
\begin{equation}\label{QDX}
Q_{D,\X}(\w):=D\tilde D M_{\rho,d}(\oo)-2\sum_{j=1}^Nw_jD M_{\rho,d}(\z-\x_j)
+\sum_{i,j=1}^Nw_iw_jM_{\rho,d}(\x_i-\x_j),
\end{equation}
with
$$
\tilde D f:=D f(-\cdot)=\sum_{|\alpha|\le k}(-1)^{|\alpha|}c_\alpha\partial^\alpha f.$$
The \emph{optimal recovery error}
$$ 
\inf_{\w\in\RR^N}\sup_{\|f\|_{H^\rho(\RR^d)}\le 1}|D f(\z)-\sum_{j=1}^Nw_jf(x_j)|=\sqrt{Q_{D,\X}(\w^*)}
$$ 
is attained by the kernel-based numerical differentiation formula with weight vector
 $\w^{*}$ generated by the Mat\'ern kernel $M_{\rho,d}$, see \cite[Section 2]{DavySchaback16}.

To circumvent the effect of rounding errors which have especially bad influence
on \eqref{QDX} due to the double differencing nature of $Q_{D,\X}(\w)$, we compute the weights $\w$ 
for small $h$ using the variable-precision arithmetic (up to 64 digits) of MATLAB Symbolic Math Toolbox.

We will compare  the errors of various formulas for the 
numerical differentiation of the Laplacian 
$Df=\Delta f=\frac{\partial^2f}{\partial x_1^2}+\frac{\partial^2f}{\partial x_2^2}$ in two variables
evaluated at the origin $\z=\oo$. We consider the worst case errors on the unit ball of the 
Sobolev space $H^\rho(\RR^2)$ evaluated by $\sqrt{Q_{D,\X}(\w)}$ derived from the kernel $M_{\rho,2}$, and  the errors 
\begin{equation}\label{testerr}
\Big|\Delta f_i(\z)-\sum_{j=1}^Nw_jf_i(\x_j)\Big|,\qquad i=1,2,
\end{equation}
for two test functions described below. 

The first test function is given by
$$
f_1(\x):=\phi_{3,2}(\|\x\|_2)(x_1+x_2)+\phi_{3,3}(\|\x\|_2),$$
where $\phi_{3,2}(r)=(1-r)_+^6(35r^2+18r+3)$ and $\phi_{3,3}(r)=(1- r)_+^8(32r^3 + 25r^2 +8r +1)$
are compactly supported radial basis function
of  Wendland's family \cite{wendland:2005-1}.
It is easy to check that 
$f_1\in C^{5,1}(\RR^2)= W^6_\infty(\RR^2)$, but its 6th order derivatives
are discontinuous at the origin, so that we consider it as a typical representative of the space $C^{5,1}(\Omega)$ 
when evaluating its derivatives at the origin. Note that $f_1\notin H^7(\RR^2)$ but nevertheless
 $f_1\in W^7_p(\RR^2)$ for all $p<2$.
 
Our second test function is
$$
f_2(\x):=e^{x_1+x_2}.$$ 
It is infinitely differentiable and it is easy to see that
$$
|f_2|_{\infty,m,\Omega}=\frac{2^{m/2}}{m!}\|f_2\|_{L^\infty(\Omega)},\quad m=0,1,\ldots,$$
and 
$$
\inf\big\{\|f_2\|_{L^\infty(\Omega)}:\Omega\supset S_{\oo,\X}\big\}=\max_{\x\in\X}f_2(\x),$$
which allows explicit computation of the error bounds that include the factor $|f|_{\infty,m,\Omega}$.

\subsection{Errors of minimal formulas of various exactness orders}\label{exp1}

Our first goal is to see how close we can get to the optimal recovery error by using minimal formulas of the
types considered in Sections~\ref{ell1} and \ref{ell2}. As already observed in the experiments in \cite{DavySchaback16}, 
weighted $\ell_1$-minimal formulas can compete rather well with the optimal recovery formulas obtained by kernel-based
numerical differentiation if their exactness order $q$ is appropriately chosen. We now confirm a similar behavior of the 
formulas generated by weighted least squares.

We consider three sets $\hat\X_i\subset[-1,1]^2$, $i=1,2,3$, introduced in \cite{DavySchaback16},
 each consisting of 32 points  containing the origin, and generate 
numerical differentiation centers by scaling %
$\X_i^h=h\hat\X_i$, where  $h=2^{-n}$, $n=0,\ldots,9$. 
The set $\hat\X_1$ consists of the origin and 31 random points in $[-1,1]^2$ drawn from the uniform 
distribution, $\hat\X_2$ includes 32 points on a straight line, a hyperbola and an
ellipse, perturbed (except of the point at the origin) randomly 
by at most $10^{-6}$ in both coordinate directions, and $\hat\X_3$ includes 32 points on three parallel
straight lines, perturbed in the same way, see Figures~1 and 2 in \cite{DavySchaback16}. 
Since 32 lies between $\dim \Pi^2_7=28$ and $\dim \Pi^2_8=36$, we consider formulas with exactness order
$q\le7$. For simplicity, we choose the weight exponent $\mu=q$, leaving the experiments with varying $\mu$ to
Section~\ref{expmu}.

In Figure~\ref{H7L1LS} we compare the worst case error \eqref{optest} of the numerical differentiation of the Laplacian on $H^7(\RR^2)$ for the 
$\dnorm_{1,q}$ and $\dnorm_{2,q}$-minimal formulas of
exactness order $q=3,\ldots,7$. Figure~\ref{f1L1LS} presents the comparison of the actual error \eqref{testerr} 
of the same formulas for the test function $f_1$. The error of the numerical differentiation using the weights
generated by the Mat\'ern kernel $M_{7,2}$ is also included, leading in Figure~\ref{H7L1LS} to the 
optimal recovery error. Note that the error of the $\dnorm_{1,7}$-minimal formulas of exactness order 7 is not shown for the
sets $\X_3^h$ because {\tt linprog} fails to compute the weight vector in this case. Nevertheless, the error of the
corresponding $\dnorm_{2,7}$-minimal formulas is very big, and the same is expected from  the $\dnorm_{1,7}$-minimal 
formulas.

\begin{figure}[htbp!]
\subfigure[$\dnorm_{1,q}$-minimal on $\X_1^h$]
{\includegraphics[width=.5\textwidth]{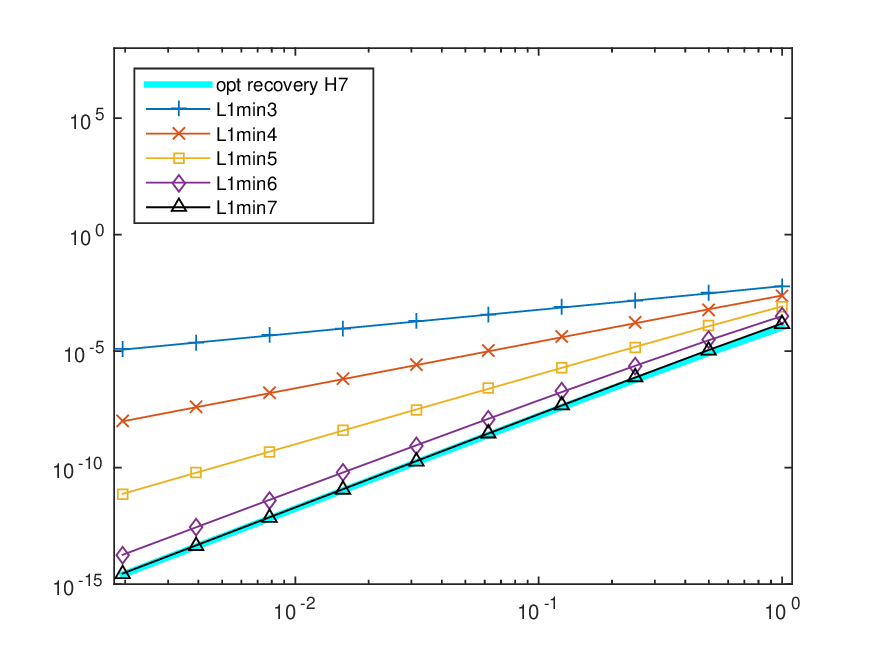}\label{random32H7L1}} 
\subfigure[$\dnorm_{2,q}$-minimal on $\X_1^h$]
{\includegraphics[width=.5\textwidth]{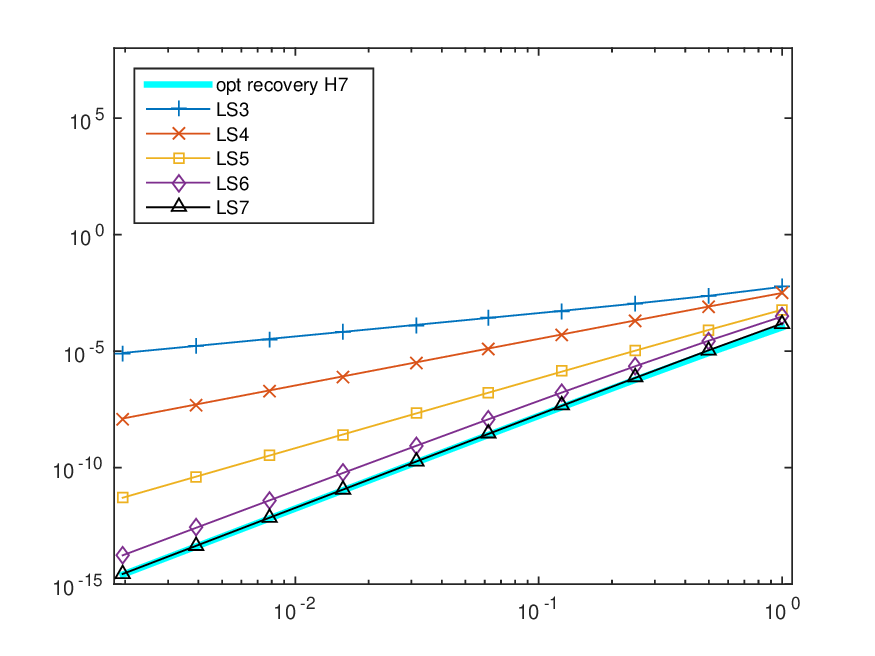}\label{random32H7LS}} 
\subfigure[$\dnorm_{1,q}$-minimal on $\X_2^h$]
{\includegraphics[width=.5\textwidth]{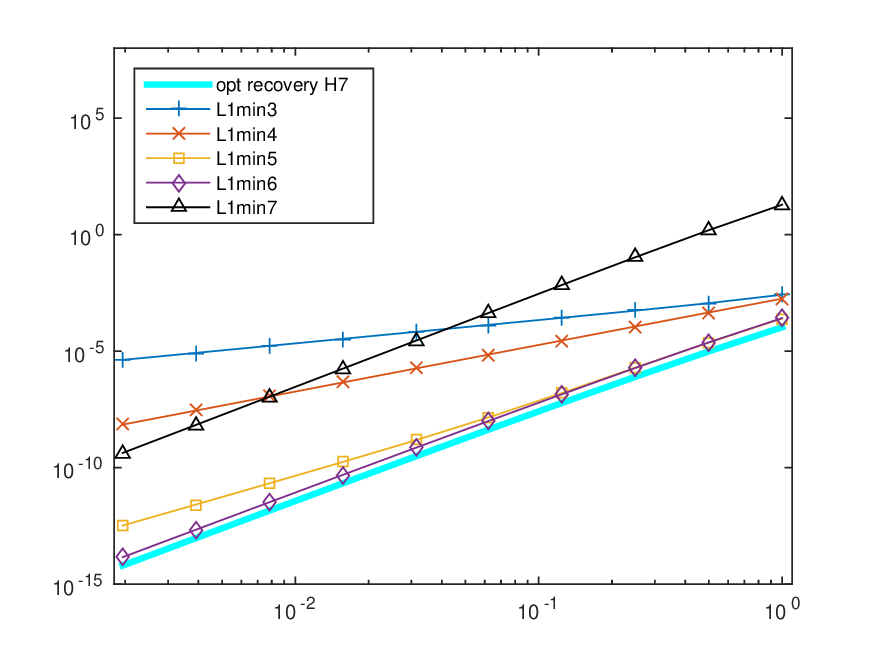}\label{she32H7L1}} 
\subfigure[$\dnorm_{2,q}$-minimal on $\X_2^h$]
{\includegraphics[width=.5\textwidth]{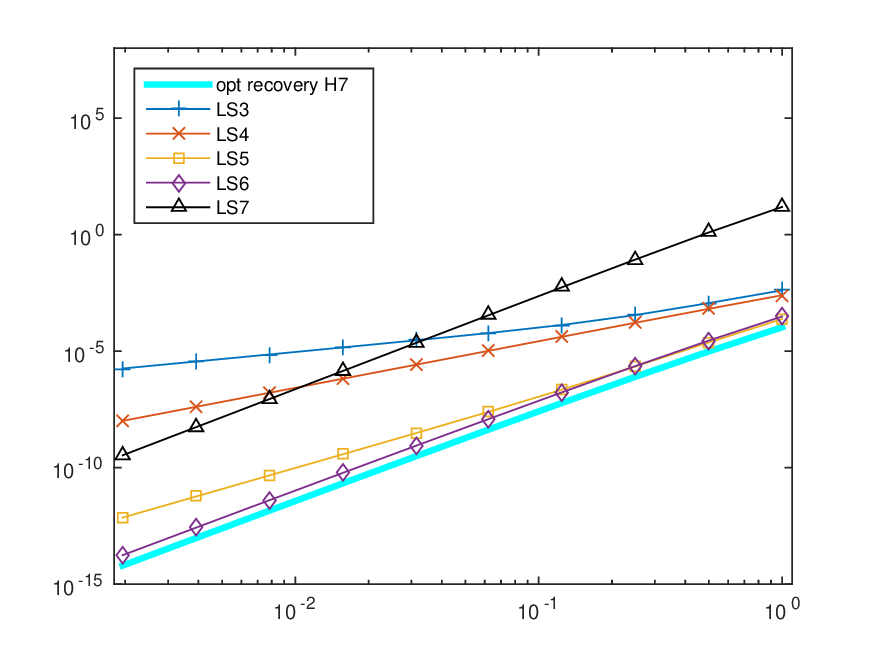}\label{she32H7LS}} 
\subfigure[$\dnorm_{1,q}$-minimal on $\X_3^h$]
{\includegraphics[width=.5\textwidth]{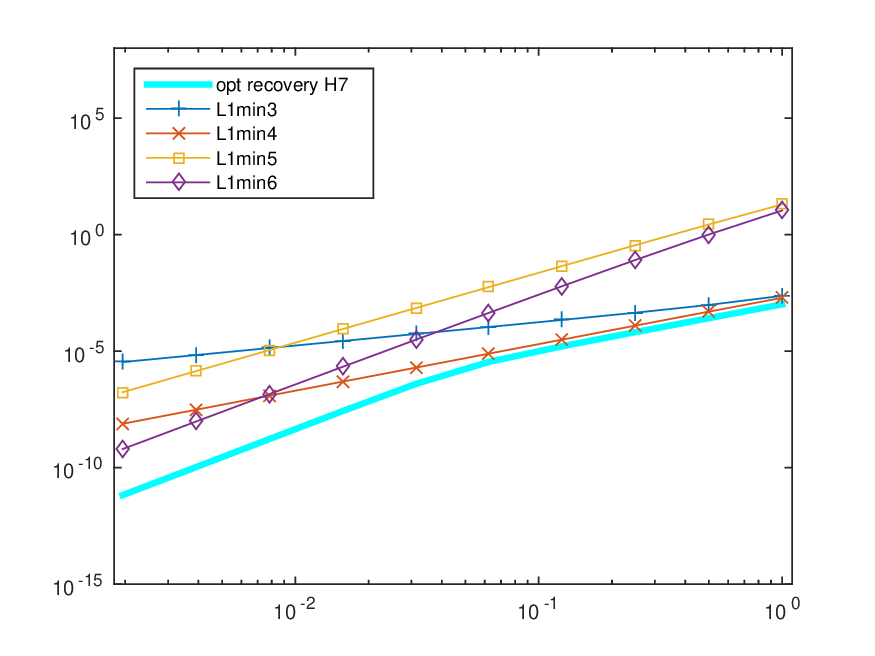}\label{3lines32H7L1}} 
\subfigure[$\dnorm_{2,q}$-minimal on $\X_3^h$]
{\includegraphics[width=.5\textwidth]{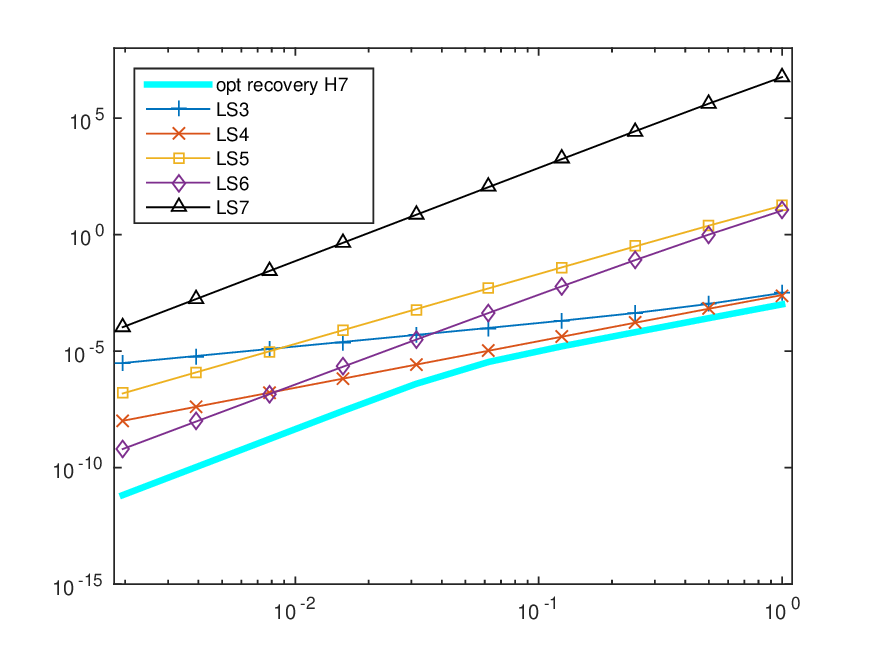}\label{3lines32H7LS}} 
\caption{Error of numerical differentiation of the Laplacian by weighted
$\ell_{1}$-minimal formulas (left) and by the least squares formulas (right)
on the Sobolev space $H^7(\RR^2)$ using centers in $\X^h_i$, $i=1,2,3$, as function of $h$.
The error of optimal recovery ({\tt opt recovery H7}) obtained with the Mat\'ern kernel $M_{7,2}$ 
is included for comparison.
{\tt L1min[q]} (resp.\ {\tt LS[q]}): $\dnorm_{1,\mu}$-minimal (resp.\ $\dnorm_{2,\mu}$-minimal)  
formula with exactness order $q$ and weight exponent $\mu=q$.
}\label{H7L1LS}
 \end{figure}

\begin{figure}[htbp!]
\subfigure[$\dnorm_{1,q}$-minimal on $\X_1^h$]
{\includegraphics[width=.5\textwidth]{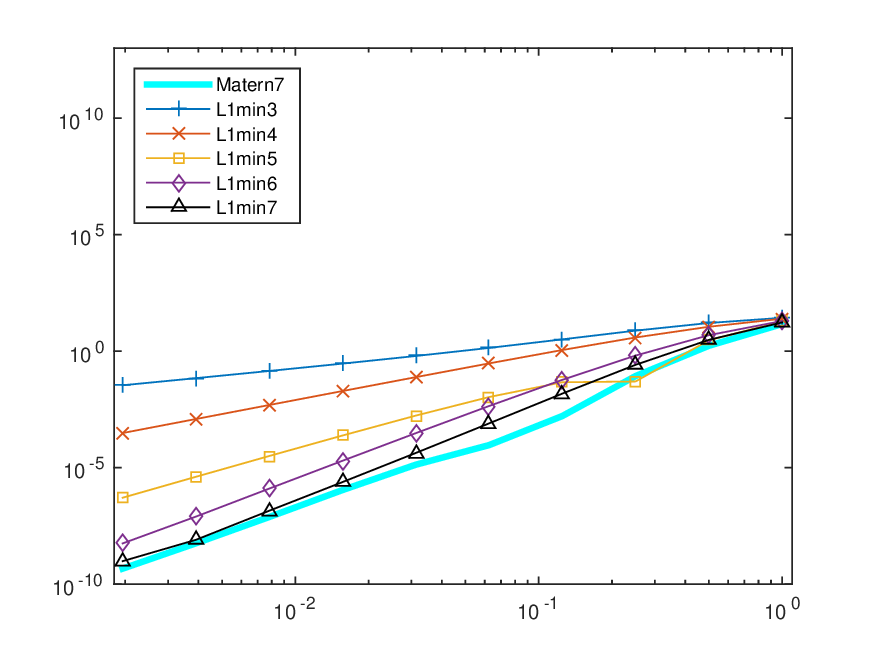}\label{random32f0L1}} 
\subfigure[$\dnorm_{2,q}$-minimal on $\X_1^h$]
{\includegraphics[width=.5\textwidth]{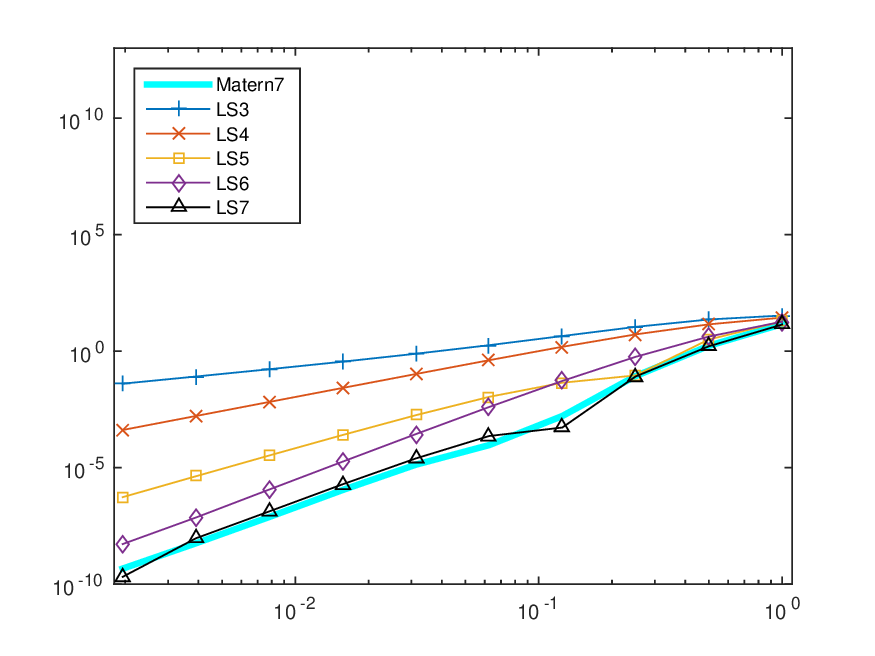}\label{random32f0LS}} 
\subfigure[$\dnorm_{1,q}$-minimal on $\X_2^h$]
{\includegraphics[width=.5\textwidth]{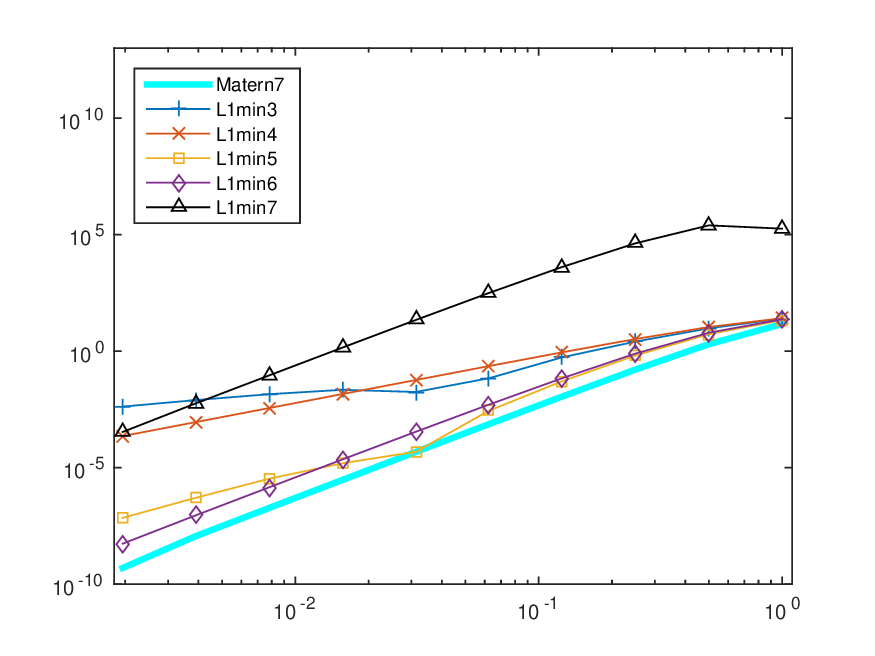}\label{she32f0L1}} 
\subfigure[$\dnorm_{2,q}$-minimal on $\X_2^h$]
{\includegraphics[width=.5\textwidth]{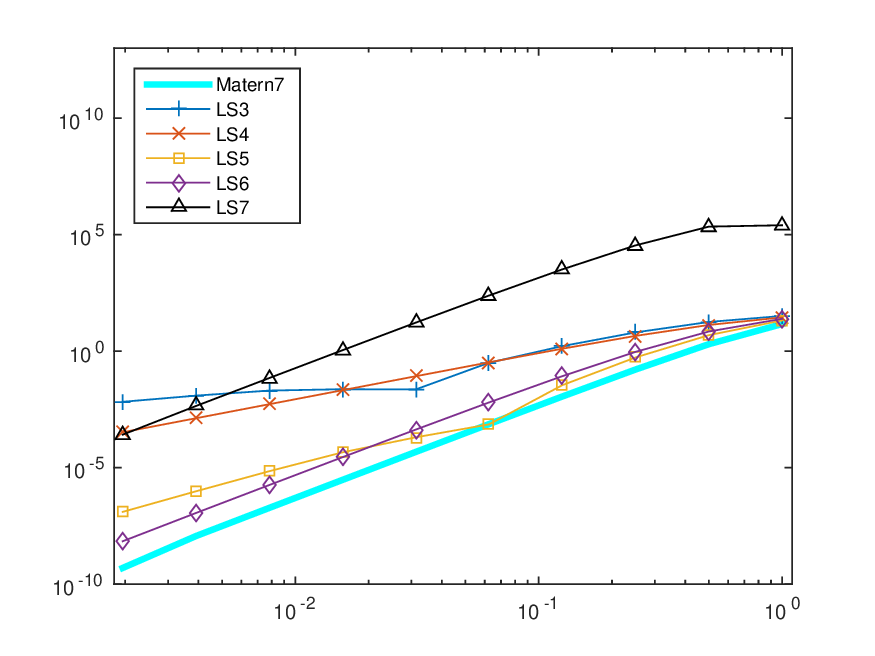}\label{she32f0LS}} 
\subfigure[$\dnorm_{1,q}$-minimal on $\X_3^h$]
{\includegraphics[width=.5\textwidth]{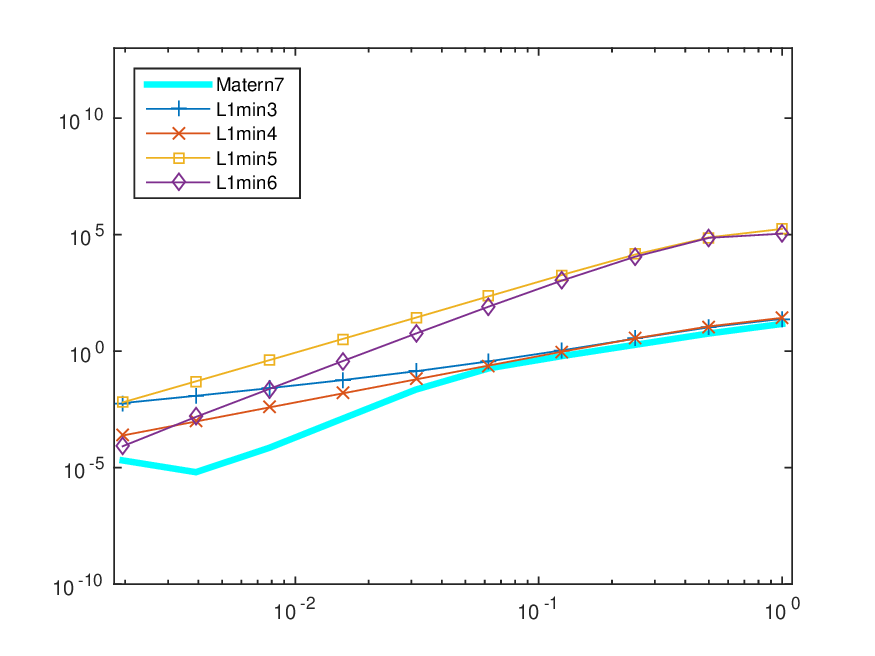}\label{3lines32f0L1}} 
\subfigure[$\dnorm_{2,q}$-minimal on $\X_3^h$]
{\includegraphics[width=.5\textwidth]{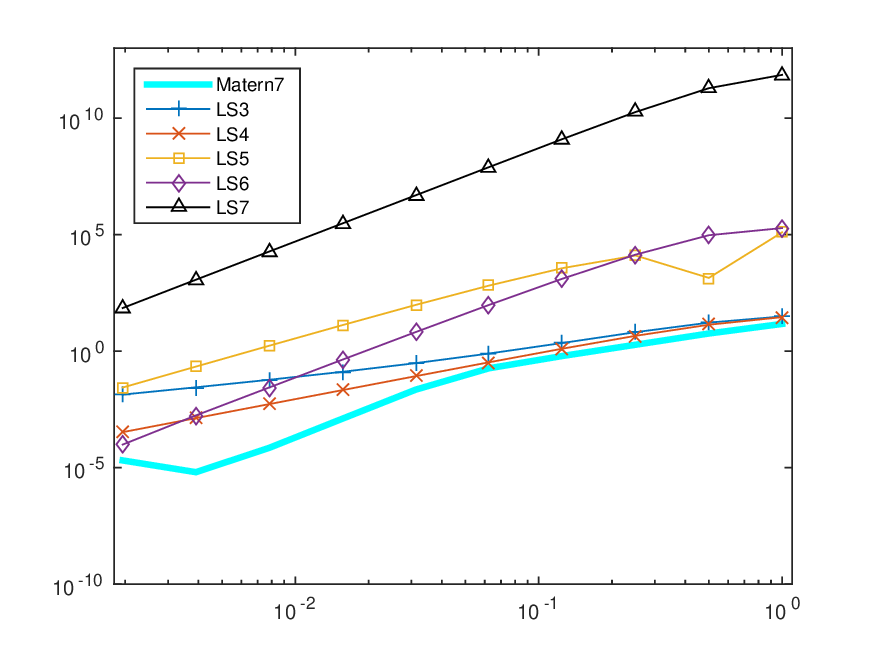}\label{3lines32f0LS}} 
\caption{Error of numerical differentiation of the Laplacian  by weighted
$\ell_{1}$-minimal formulas (left) and by the least squares formulas (right)
for the test function $f_1$ using centers in $\X^h_i$, $i=1,2,3$, as function of $h$.
The curves correspond to the weights obtained by different methods. 
The error of the optimal recovery weights for the space $H^7(\RR^2)$ obtained with the Mat\'ern kernel $M_{7,2}$ 
is included for comparison ({\tt Matern7}).
{\tt L1min[q]} (resp.\ {\tt LS[q]}): $\dnorm_{1,\mu}$-minimal (resp.\ $\dnorm_{2,\mu}$-minimal)  
formula with exactness order $q$ and weight exponent $\mu=q$.
}\label{f1L1LS}
 \end{figure}

The results indicate that the errors of $\dnorm_{1,q}$-minimal and $\dnorm_{2,q}$-minimal formulas are close for both the worst case on $H^7(\RR^2)$ and for the test
function $f_1$. The best  convergence order as $h\to0$ achieved in these experiments is $h^4$ for all sets $\X^h_i$, $i=1,2,3$, and it is attained by the formulas of
polynomial exactness order $q=6$ or 7. 
The fact that the formulas of exactness order 7 do possess better convergence speed than those of order 6 is explained by the  finite smoothness of $f_1$ and 
$H^7(\RR^2)$, see  \cite{DavySchabackOPT}. There are significant differences between the sets $\X^h_1$, $\X^h_2$ and $\X^h_3$ with respect to the question which $q$ is
preferable, especially in the pre-asymptotic setting where $h$ is not
excessively small. For the random set $\X^h_1$ the largest possible
$q=7$ is the best for all $h$,
whereas it leads to very big errors for the other two sets $\X^h_2,\X^h_3$ where the points are placed close to algebraic curves of degree 5, respectively 3. The best
choice for $\X^h_2$ is $q=6$, with the errors of $q=5$ quite competitive for larger $h$. The situation is more complicated for $\X^h_3$, where $q=4$ is the best choice for
larger $h$, and the errors are comparable to those of the optimal recovery, whereas $q=6$ becomes better for smaller $h$, but the errors are nevertheless
significantly worse than  the optimal recovery. Note that the slope of the error of the Mat\'ern kernel changes, as $h$ decreases, from $h^2$ matching the slope of the error curve for
$q=4$, to  $h^4$ matching the slope of the error curve for $q=6$.

These results emphasize the need for a careful selection of the exactness order $q$ of a numerical differentiation formula,
which cannot always be made only on the basis of the number of points in the set $\X$, or their nearly uniform distribution.
(The set $\X_2$ fills out the square $[0,1]^2$ more uniformly than $\X_1$ does, see \cite[Figure 1]{DavySchaback16}.)  To see
how well our estimates in Sections~\ref{ell1} and \ref{ell2} can predict which $q$ leads to smaller errors, we present in
Figure~\ref{f2L1LS} the errors of the $\dnorm_{1,q}$-minimal and $\dnorm_{2,q}$-minimal formulas of exactness order $q$ for
the numerical differentiation of the Laplacian of the test function $f_2$,  together with the estimates of this error provided
by the inequalities \eqref{Cpeb} and \eqref{lsbwoa}/\eqref{peb1s1}. 
(Note that in these estimates we take the infimum over all $\Omega$ containing the convex hull of $\{\z,\x_1,\ldots,\x_N\}$.) 
Note that the estimate \eqref{peb1a} for the
$\dnorm_{1,q}$-minimal formulas  of exactness order $q$ coincides with \eqref{Cpeb}. The error of the $\dnorm_{1,7}$-minimal
formulas of exactness order 7 is not shown in Figure~\ref{f2L1LS}(e) because  {\tt linprog} fails to compute their weight
vectors, as mentioned above for Figures~\ref{H7L1LS} and \ref{f1L1LS}. 

The results in Figure~\ref{f2L1LS} show that both bounds \eqref{Cpeb} and \eqref{lsbwoa}/\eqref{peb1s1} correctly predict 
$q$ with the smallest error in most cases. However, the estimates seem to become less effective as $q$ increases.

\begin{figure}[htbp!]
\subfigure[$\dnorm_{1,q}$-minimal on $\X_1^h$]
{\includegraphics[width=.5\textwidth]{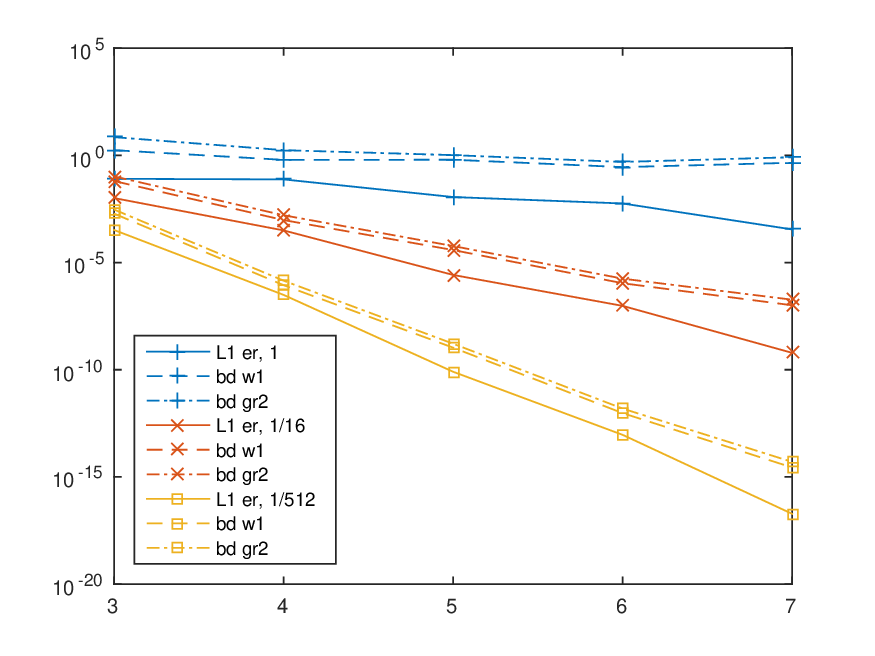}\label{random32f2L1}} 
\subfigure[$\dnorm_{2,q}$-minimal on $\X_1^h$]
{\includegraphics[width=.5\textwidth]{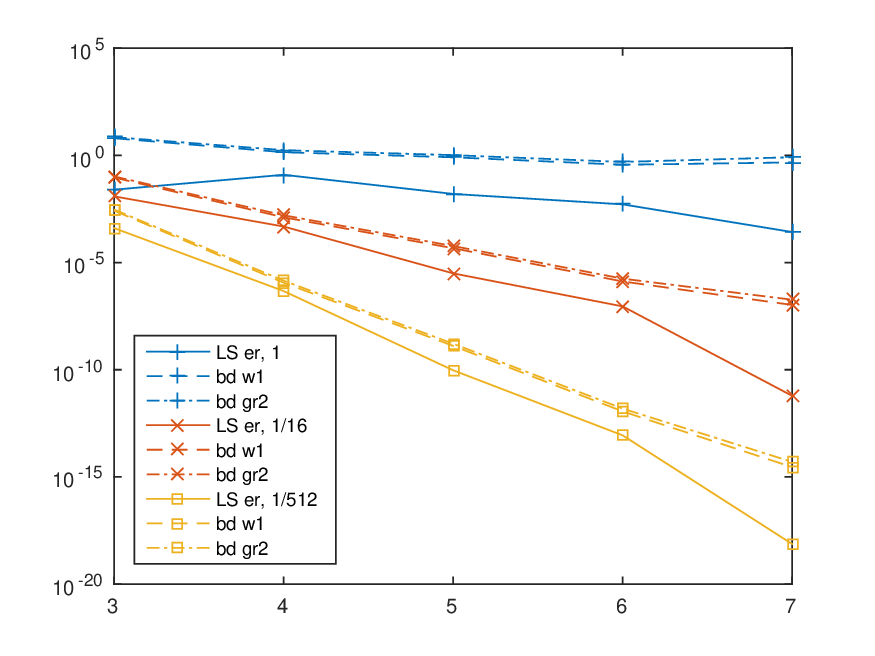}\label{random32f2LS}} 
\subfigure[$\dnorm_{1,q}$-minimal on $\X_2^h$]
{\includegraphics[width=.5\textwidth]{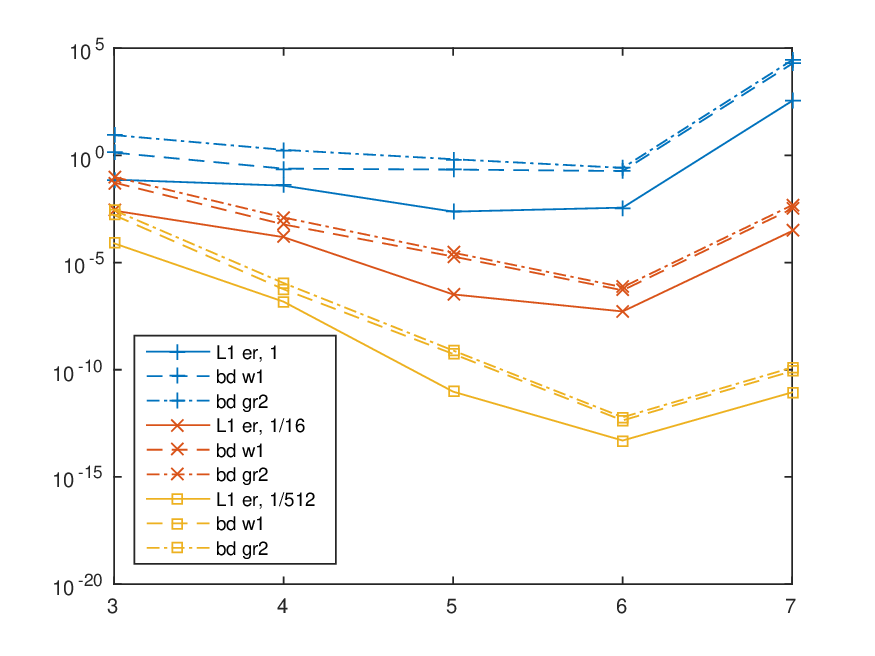}\label{she32f2L1}} 
\subfigure[$\dnorm_{2,q}$-minimal on $\X_2^h$]
{\includegraphics[width=.5\textwidth]{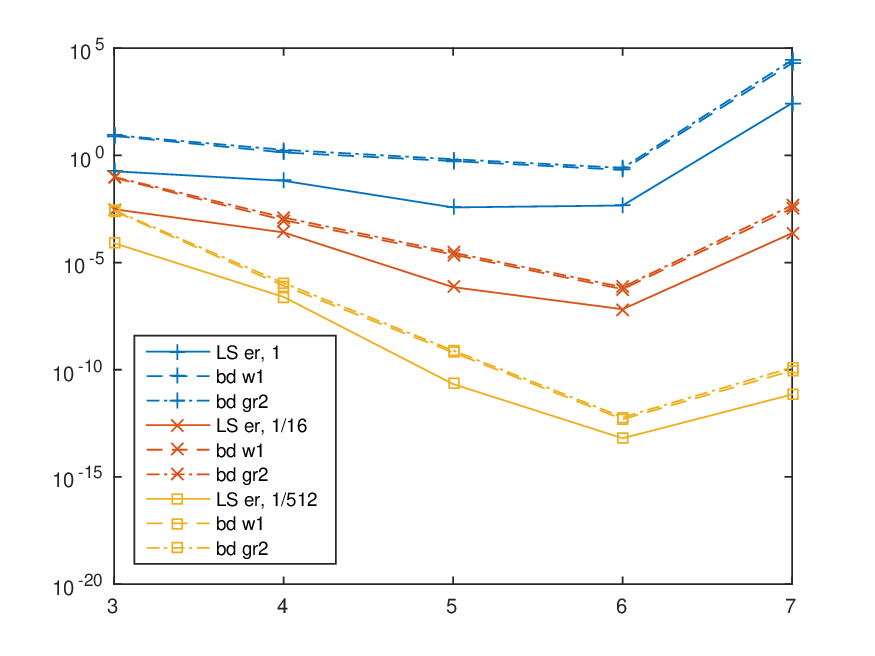}\label{she32f2LS}} 
\subfigure[$\dnorm_{1,q}$-minimal on $\X_3^h$]
{\includegraphics[width=.5\textwidth]{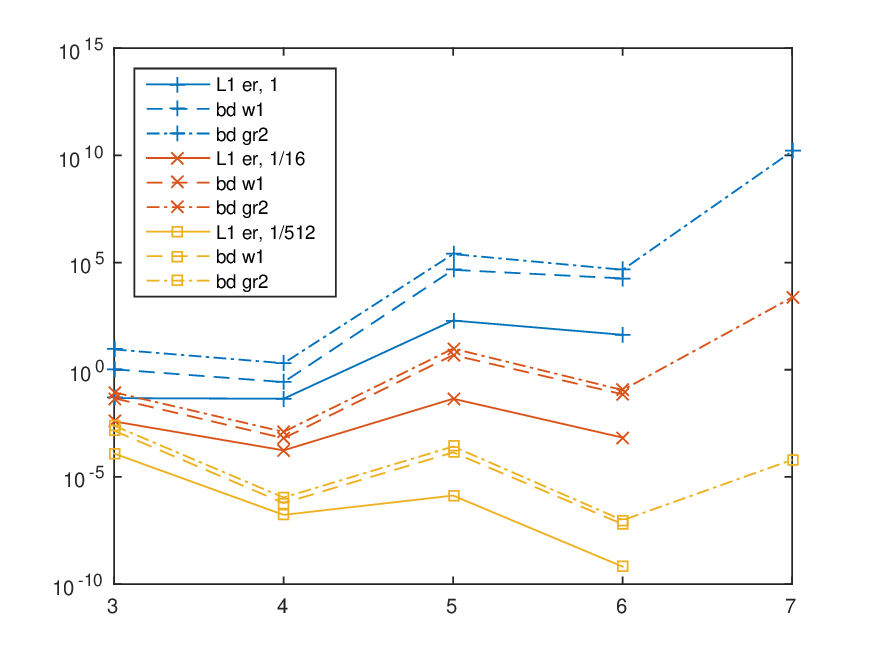}\label{3lines32f2L1}} 
\subfigure[$\dnorm_{2,q}$-minimal on $\X_3^h$]
{\includegraphics[width=.5\textwidth]{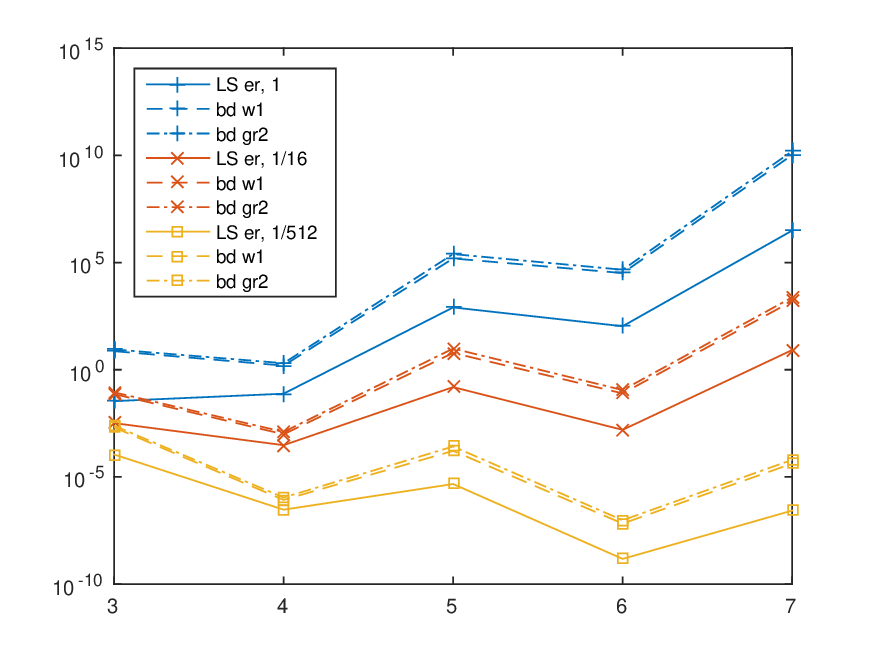}\label{3lines32f2LS}} 
\caption{Error of numerical differentiation of the Laplacian by  weighted
$\dnorm_{1,q}$-minimal formulas (left) and $\dnorm_{2,q}$-minimal formulas (right)
of exactness order $q$ for the test function $f_2$ using centers in $\X^h_i$, $i=1,2,3$, $h=1,1/16,1/512$, 
as functions of $q=3,\ldots,7$, together with error bounds.
{\tt L1 er, h}: the error of the  $\dnorm_{1,q}$-minimal formula of exactness order $q$;
{\tt LS er, h}: the error of the  $\dnorm_{2,q}$-minimal formula of exactness order $q$;
{\tt bd w1}: error bound \eqref{Cpeb}; 
{\tt bd g2}: error bound \eqref{lsbwoa}/\eqref{peb1s1} for both $\dnorm_{1,q}$ and
$\dnorm_{2,q}$-minimal formulas.
}\label{f2L1LS}
 \end{figure}

\subsection{Influence of the exponent $\mu$}
\label{expmu}

In the above experiments we used $\mu=q$ for formulas of exactness order $q$. In order to see how the choice of $\mu$ 
influences the results, 
we consider two new sets $\X_4$ and $\X_5$ shown in Figure~\ref{X45}.
The set $\X_4$ consists of the origin $(0,0)$ and 149 random points from the uniform distribution in $[-1,1]$. 
To obtain $\X_5$, we removed from $\X_4$ its 32 points closest to $(0,0)$, and
replaced them by the set $\X_3^h$ with $h=1/\sqrt{2}$. This makes the central part of $\X_5$ less suitable for a
high order approximation. 

\begin{figure}[htbp!]
\subfigure[$\X_4$]
{\includegraphics[width=.5\textwidth]{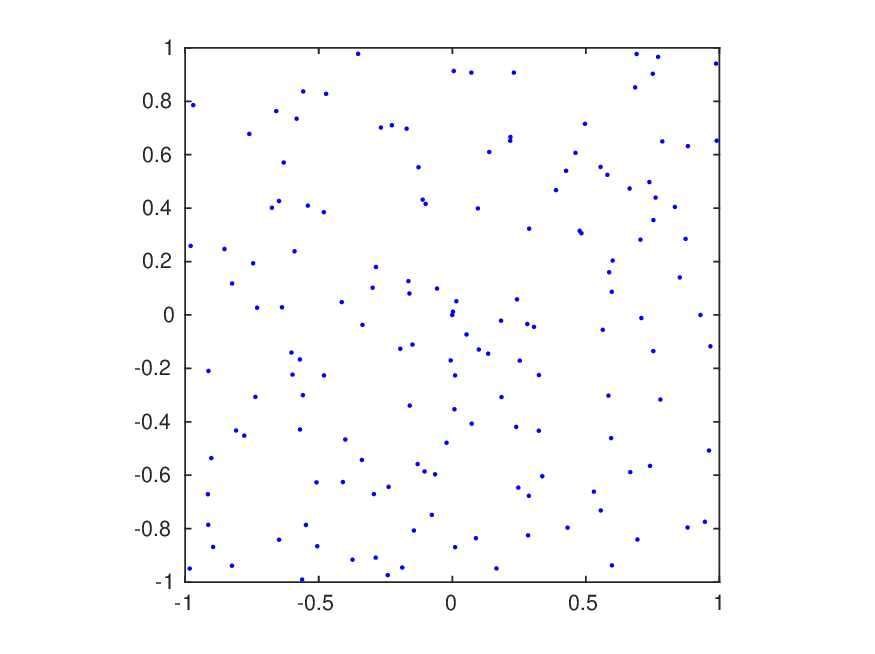}\label{points_random150}} 
\subfigure[$\X_5$]
{\includegraphics[width=.5\textwidth]{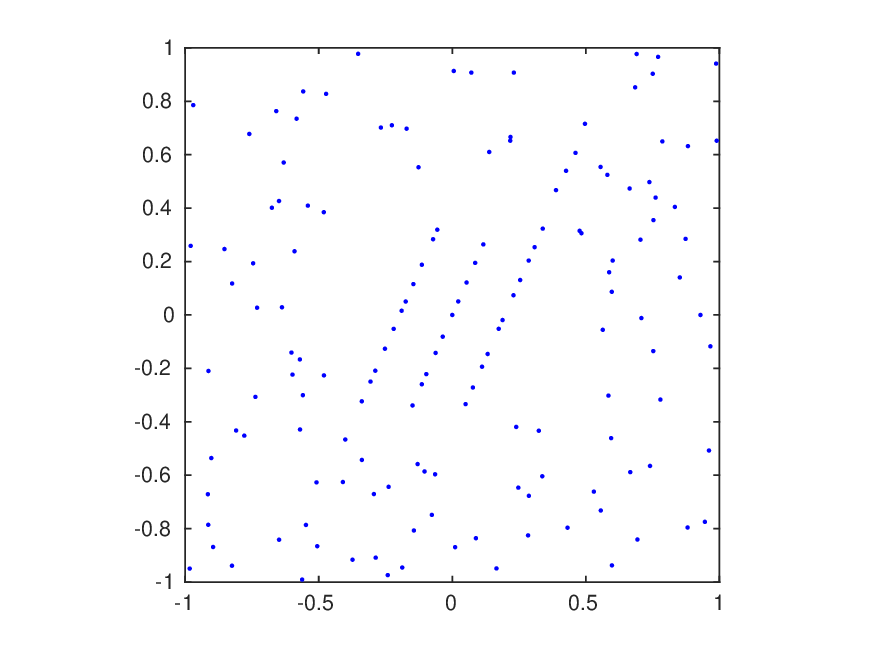}\label{points_random150with3lines}} 
\caption{Point sets  $\X_4$ (the origin and 149 random points in $[-1,1]$) and $\X_5$ obtained from $\X_4$ by replacing
its 32 points closest to $(0,0)$ by $\X_3^{1/\sqrt{2}}$.
}\label{X45}
 \end{figure}

We now compare $\dnorm_{1,\mu}$ and $\dnorm_{2,\mu}$-minimal formulas of a given
exactness order $q=7$ on the sets $\X_4$ and $\X_5$, with the  weight exponents $\mu$ varying between 0 and 15.
Figures~\ref{X45rH7}(ab) show the worst case error \eqref{optest} of these formulas on 
$H^8(\RR^2)$,  Figures~\ref{X45rH7}(cd) shows the value of the factor 
$\|\w\|_{1,q}=\sum_{j=1}^{N} |w_j|\|\x_j-\z\|_2^{q}$
in the error bound \eqref{Cpeb}, whereas Figures~\ref{X45rH7}(ef) depict the stability constant
$\|\w\|_1=\sum_{j=1}^{N} |w_j|$
responsible for the numerical stability of the respective differentiation formula. 
In addition, Figure~\ref{X45f2L1LS} presents the errors of the same formulas for test function $f_2$, together with the error
bounds  given by \eqref{Cpeb}, \eqref{peb1}, \eqref{peb2}, \eqref{peb3}, \eqref{lsb} and \eqref{peb123s1}.
Finally, Figure~\ref{X45stencils} shows which points have nonzero weights in the $\dnorm_{1,\mu}$-minimal formulas 
for the weight exponents $n=0,7,15$.

\begin{figure}[htbp!]
\subfigure[$\X_4$: error on $H^8(\RR^2)$]
{\includegraphics[width=.5\textwidth]{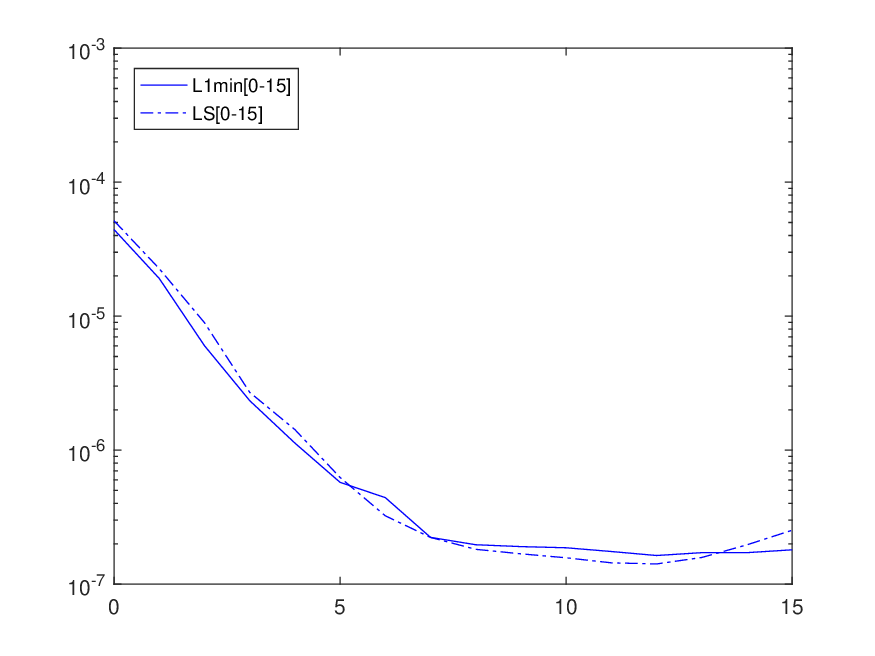}\label{random150orH8}} 
\subfigure[$\X_5$: error on $H^8(\RR^2)$]
{\includegraphics[width=.5\textwidth]{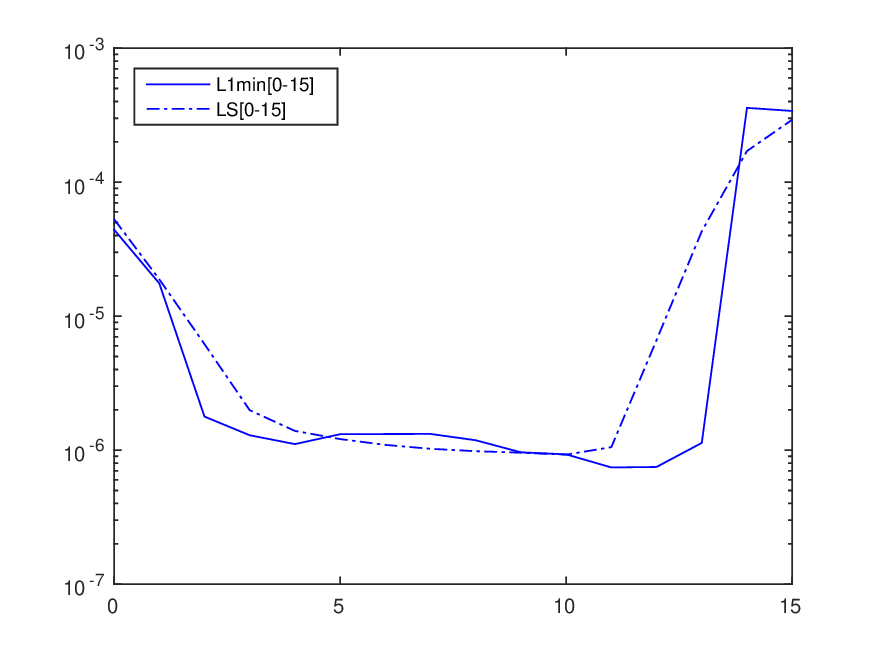}\label{random150with3lines_orH8}} 
\subfigure[$\X_4$: the value of $\|\w\|_{1,q}$]
{\includegraphics[width=.5\textwidth]{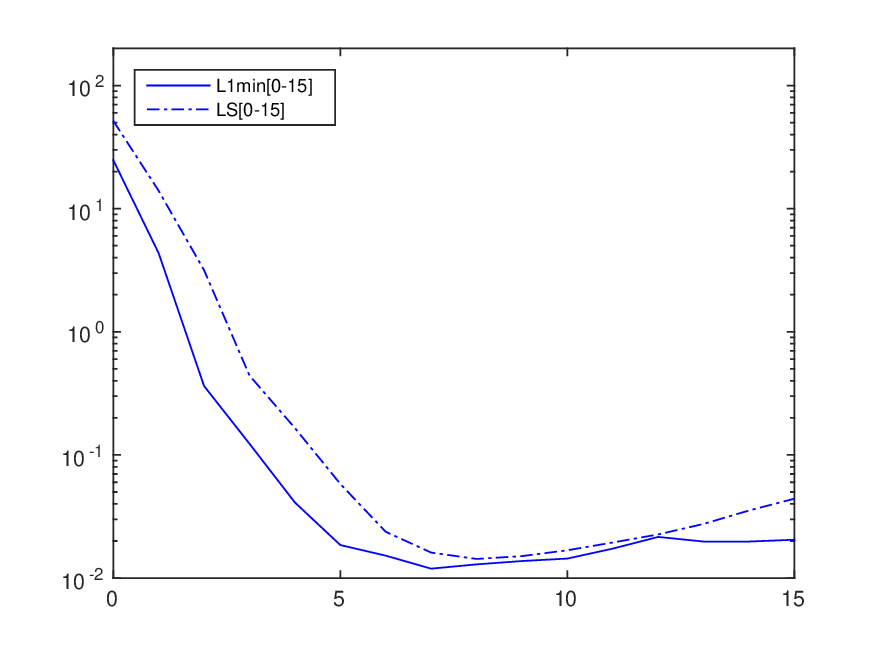}\label{random150const_q7}} 
\subfigure[$\X_5$: the value of $\|\w\|_{1,q}$]
{\includegraphics[width=.5\textwidth]{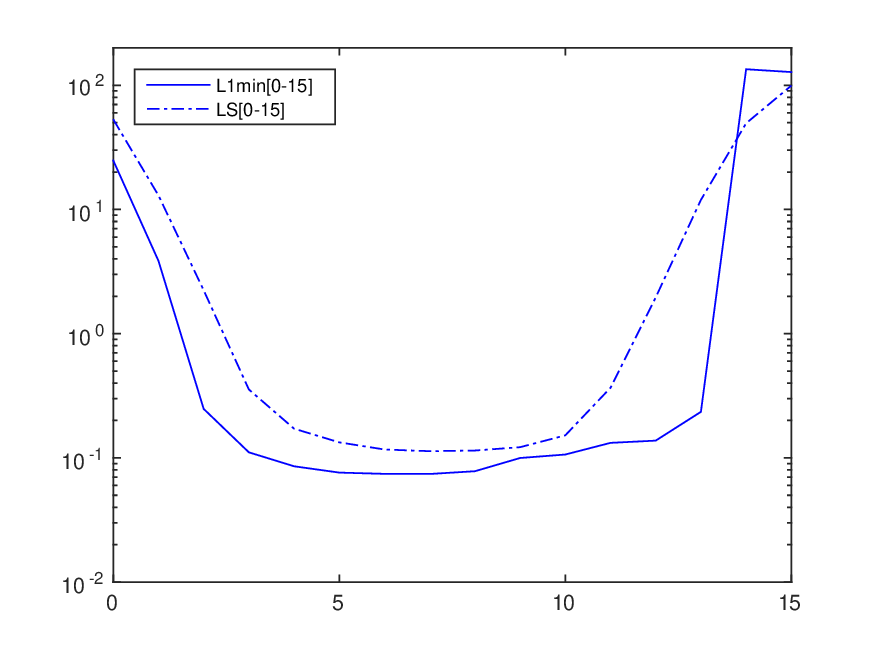}\label{random150with3lines_const_q7}} 
\subfigure[$\X_4$: stability constant $\|\w\|_1$]
{\includegraphics[width=.5\textwidth]{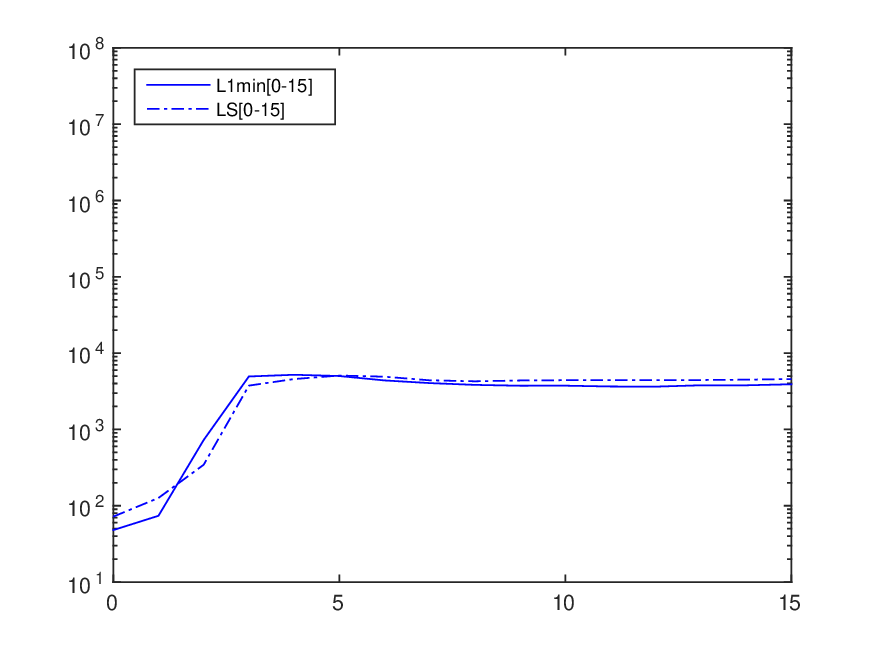}\label{random150stab_q7}} 
\subfigure[$\X_5$: stability constant $\|\w\|_1$]
{\includegraphics[width=.5\textwidth]{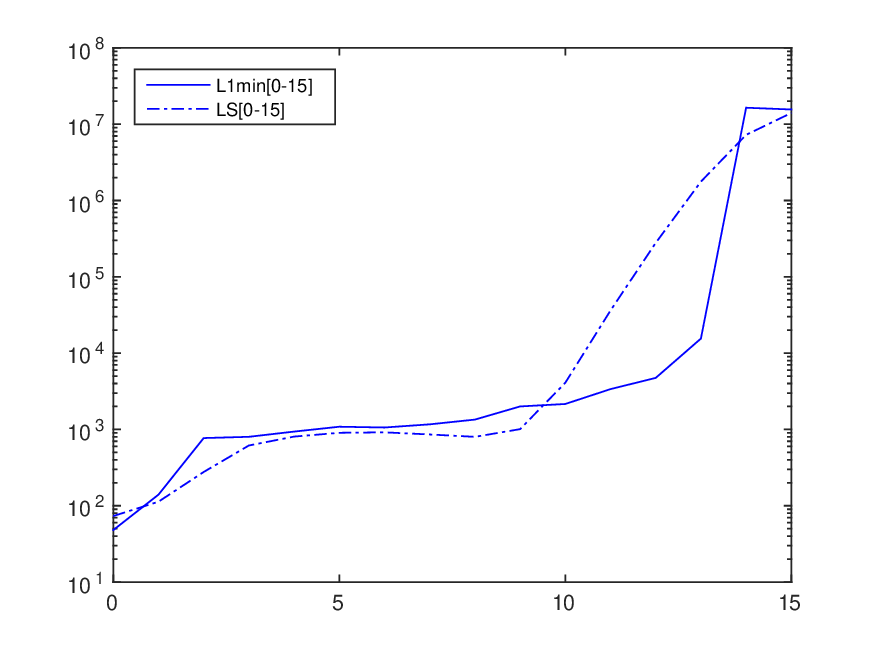}\label{random150with3lines_stab_q7}} 
\caption{Comparison of the error of $\dnorm_{1,\mu}$ and $\dnorm_{2,\mu}$  minimal formulas  of 
exactness order $q=7$ for the numerical differentiation of
the Laplacian at the origin  on the sets $\X_4$ and $\X_5$ of 
Figure~\ref{X45}, as function of $\mu=0,\ldots,15$.
}\label{X45rH7}
 \end{figure}

\begin{figure}[htbp!]
\subfigure[$\dnorm_{1,\mu}$ formulas on $\X_4$]
{\includegraphics[width=.5\textwidth]{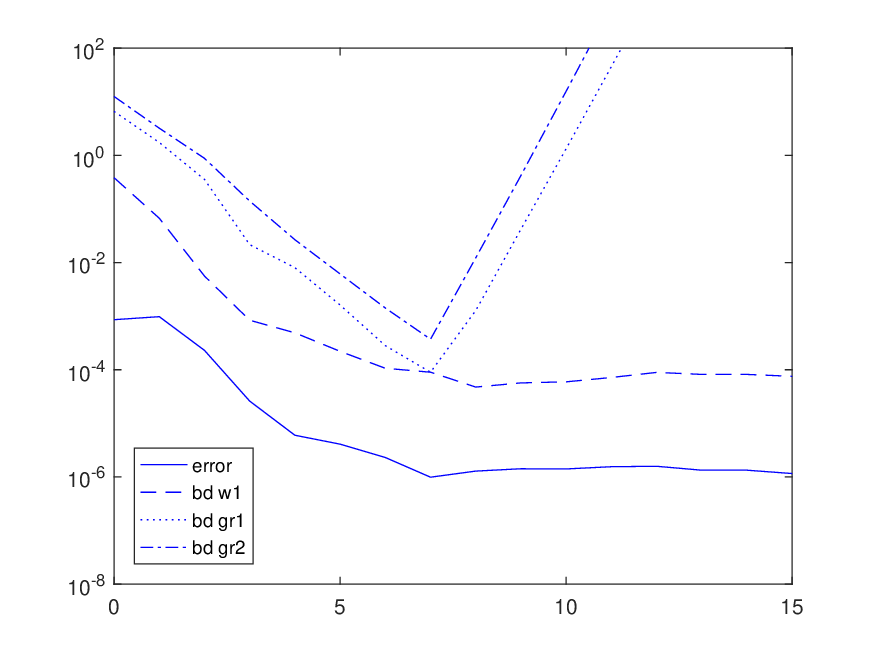}\label{random150f2L1}} 
\subfigure[$\dnorm_{1,\mu}$ formulas on $\X_5$]
{\includegraphics[width=.5\textwidth]{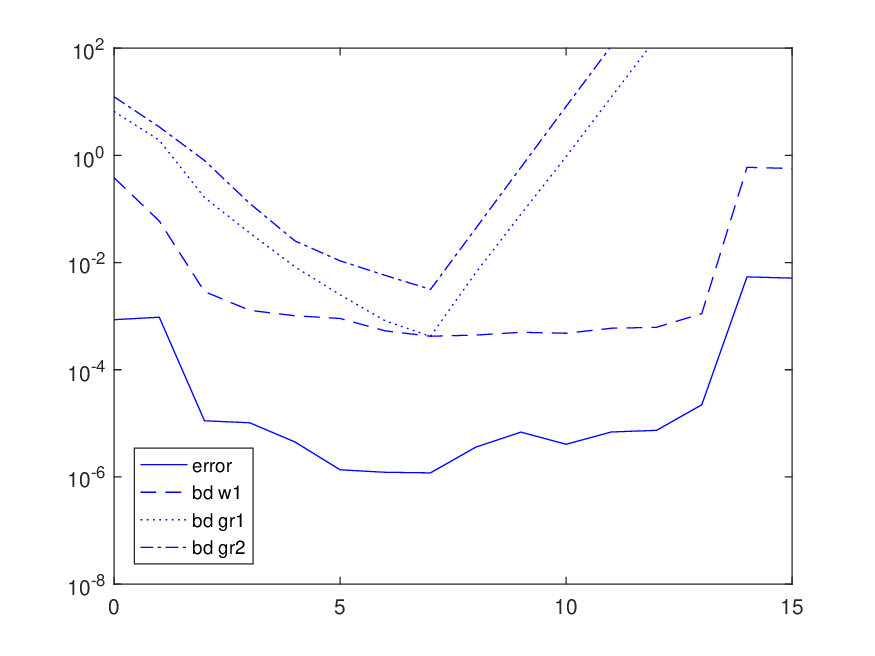}\label{random150with3lines_f2L1}} 
\subfigure[$\dnorm_{2,\mu}$ formulas on $\X_4$]
{\includegraphics[width=.5\textwidth]{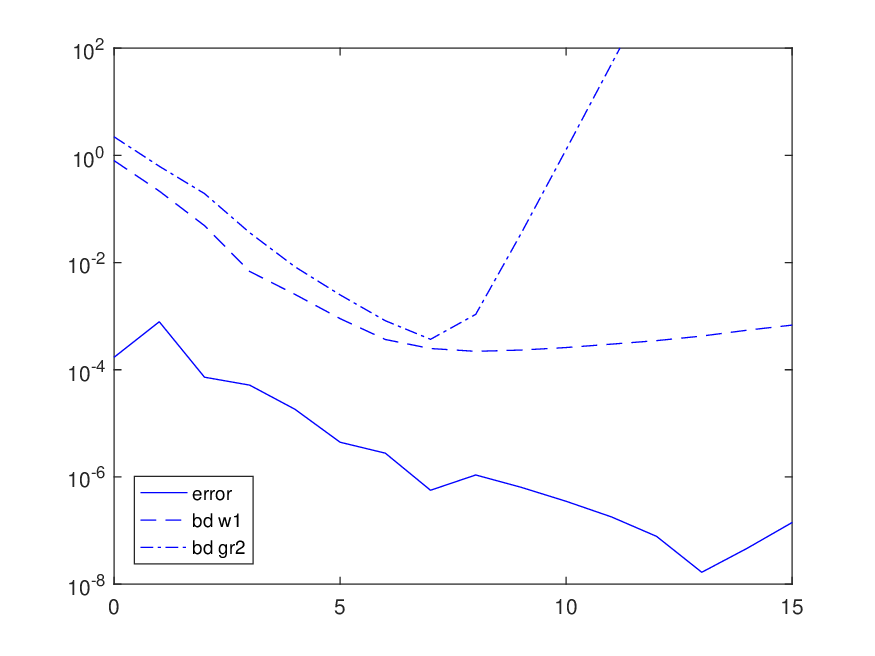}\label{random150f2LS}} 
\subfigure[$\dnorm_{2,\mu}$ formulas on $\X_5$]
{\includegraphics[width=.5\textwidth]{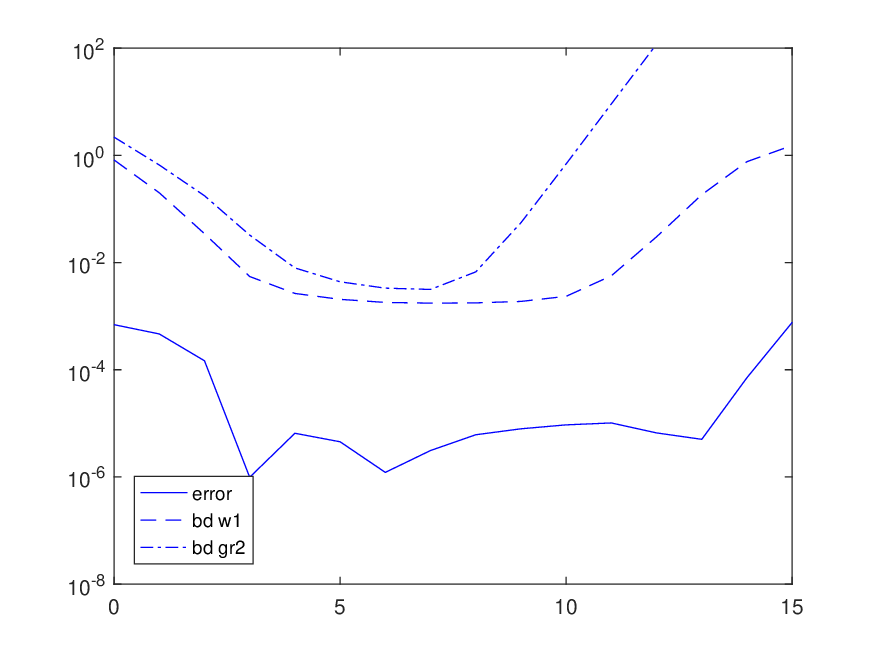}\label{random150with3lines_f2LS}} 
\caption{Error of $\dnorm_{1,\mu}$ and $\dnorm_{2,\mu}$  minimal formulas 
of exactness order $q=7$ for the numerical differentiation of
the Laplacian at the origin for test function $f_2$ on the centers in $\X_4$ and $\X_5$, together with various error bounds, 
as functions of $\mu=0,\ldots,15$.
{\tt bd w1}: error bound \eqref{Cpeb} with $\Omega$ chosen as the convex hull of $\z\cup\X$;
{\tt bd gr1}: error bound \eqref{peb1} if $\mu=q$, \eqref{peb2} if $\mu<q$ or \eqref{peb3} if $\mu>q$;
{\tt bd gr2}: error bound \eqref{peb123s1} for $\dnorm_{1,\mu}$-minimal formulas, or
\eqref{lsb} for $\dnorm_{2,\mu}$-minimal formulas.
}\label{X45f2L1LS}
 \end{figure}

\begin{figure}[htbp!]
\subfigure[$\X_4$: $\mu=0$]
{\includegraphics[width=.5\textwidth]{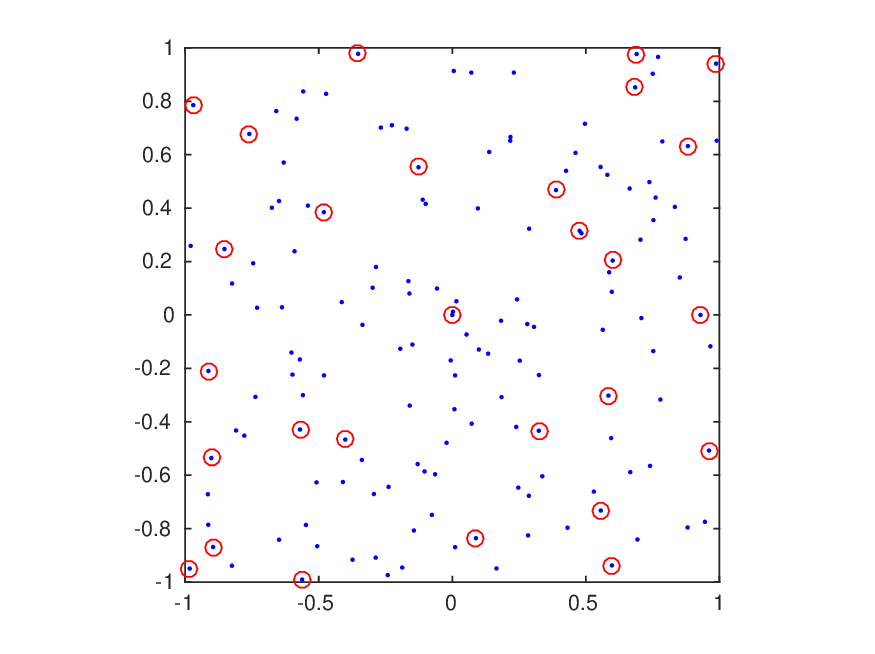}\label{points_random150p0}} 
\subfigure[$\X_5$: $\mu=0$]
{\includegraphics[width=.5\textwidth]{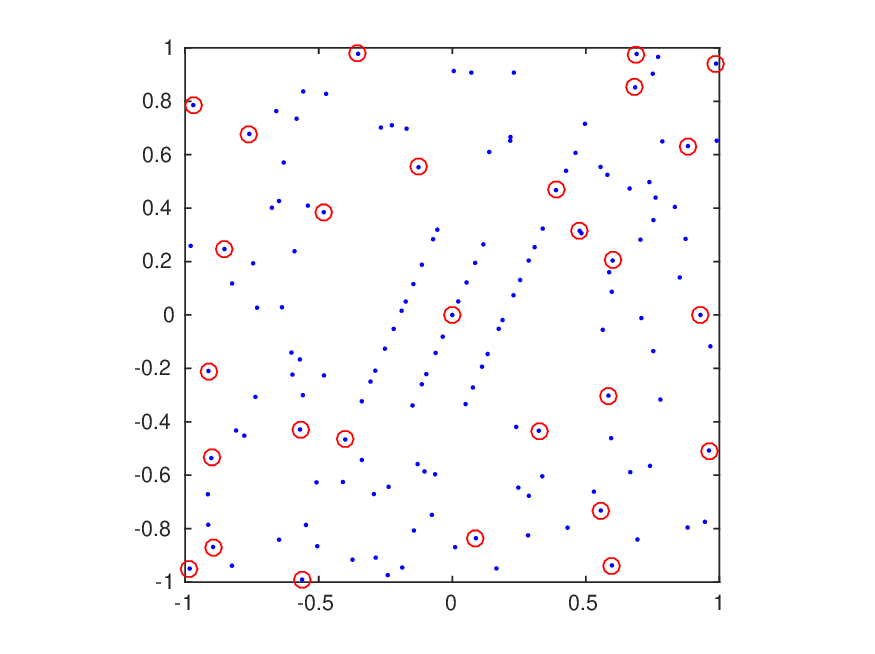}\label{points_random150with3lines_p0}} 
\subfigure[$\X_4$: $\mu=7$]
{\includegraphics[width=.5\textwidth]{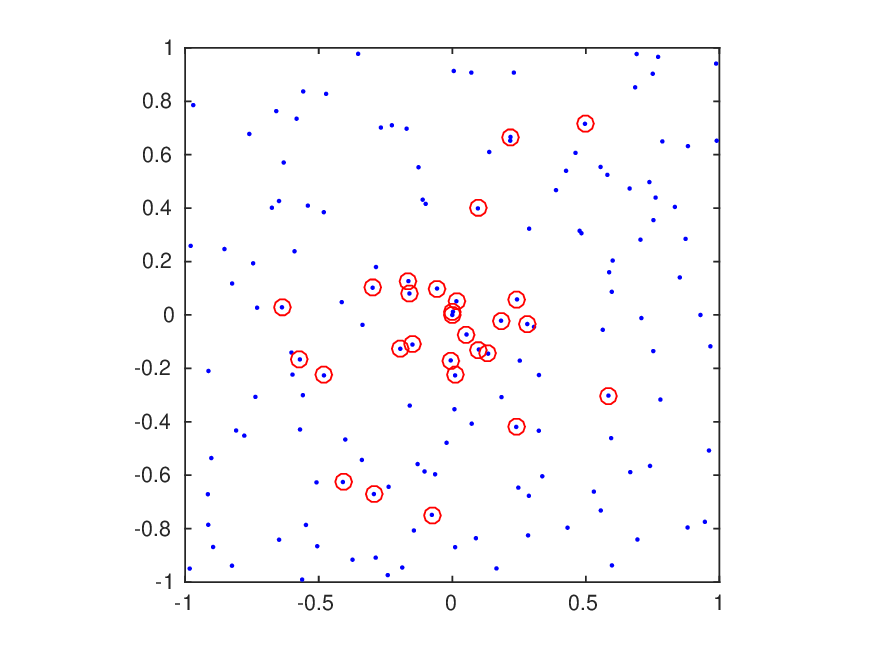}\label{points_random150p7}} 
\subfigure[$\X_5$: $\mu=7$]
{\includegraphics[width=.5\textwidth]{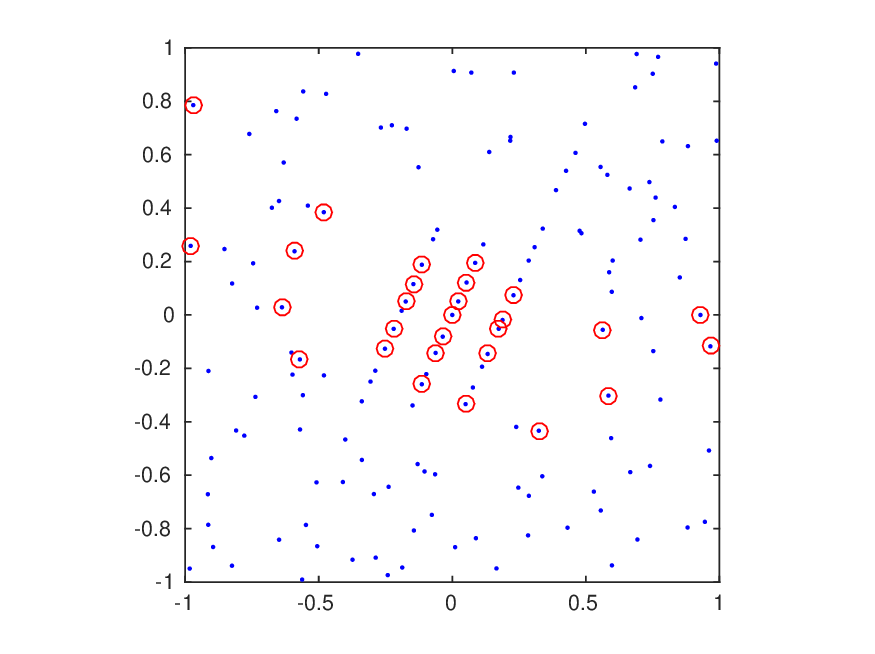}\label{points_random150with3lines_p7}} 
\subfigure[$\X_4$: $\mu=15$]
{\includegraphics[width=.5\textwidth]{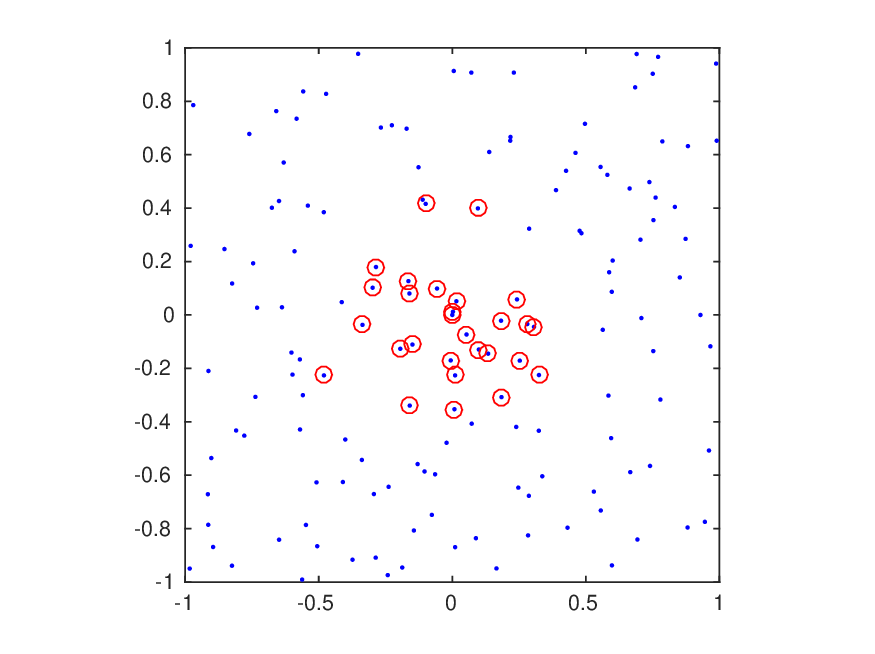}\label{points_random150p15}} 
\subfigure[$\X_5$: $\mu=15$]
{\includegraphics[width=.5\textwidth]{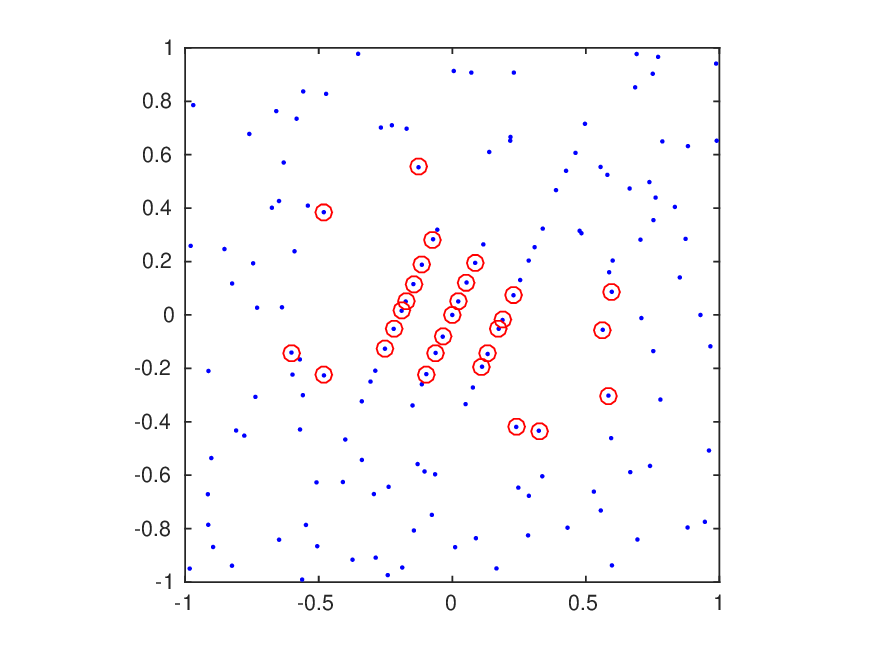}\label{points_random150with3lines_p15}} 
\caption{The 28 points of $\dnorm_{1,\mu}$-minimal formulas for $\X_4$ and $\X_5$ of exactness order $q=7$ with weight exponents 
$\mu=0,7,15$.
}\label{X45stencils}
 \end{figure}

We list some observations from these experiments:
\begin{enumerate}

\item The errors of $\ell_1$-minimal and least squares formulas behave quite similar.

\item For the random set  $\X_4$ the error generally decreases with the exponent $\mu$, see 
Figures~\ref{X45rH7}(a) and \ref{X45f2L1LS}(ac). This can be explained by the fact that a higher
exponent $\mu$ penalizes more distant points and thus forces the algorithm to choose points located 
closer to $\z$, compare Figures~\ref{X45stencils}(ace). 
The improvement of the error is most significant when $\mu$ increases from $\mu=0$ to $\mu\approx q$. 

\item For the set  $\X_5$ with a difficult configuration in the central part the errors first improve when
$\mu$ increases but start increasing when $\mu$ gets a little higher than $q$, see Figures~\ref{X45rH7}(b)
and \ref{X45f2L1LS}(bd). 
Thus, the high exponents force the algorithm to choose too many points in the 
problematic central area of  $\X_5$, see Figure~\ref{X45stencils}(f).

\item For both sets $\X_4$ and $\X_5$ the choice $\mu=q$ gives almost optimal error. It is in particular 
remarkable to see in Figure~\ref{X45stencils}(d) that the algorithm chooses more distant points of $\X_5$
(not belonging to $\X_3^{1/\sqrt{2}}$) in the direction orthogonal to the parallel lines used to generate $\X_3$,
which is heuristically a good way to compensate for the deficiency of $\X_3^{1/\sqrt{2}}$ which has too many points
in rows going  in the direction of these lines and does not provide enough information about the behavior of the test 
function in the orthogonal direction.

\item Figures~\ref{X45rH7}(cd), in comparison to \ref{X45rH7}(ab) show that the quantity $\|\w\|_{1,q}$ predicts  
very well which formulas are more accurate. The curves marked {\tt bd w1} in Figure~\ref{X45f2L1LS} confirm this 
observation by comparing the error bound \eqref{Cpeb} based on $\|\w\|_{1,q}$ to the actual error of the Laplacian for test function $f_2$.

\item A comparison of Figures~\ref{X45rH7}(ab) with Figures~\ref{X45rH7}(ef) demonstrates a trade-off 
between higher accuracy and stability: a smaller error is normally achieved at 
the expense of a larger stability constant $\|\w\|_1$. Using $\mu=q$ seems to give a good compromise in this respect, as
a significantly better stability is only obtained with small values of $\mu$ by picking more distant centers as in 
Figures~\ref{X45stencils}(ab). An excessively large $\mu$ leads on $\X_5$ to
formulas that are inaccurate
and unstable at 
the same time.

\item The curves marked {\tt bd gr1} and {\tt bd gr2} in Figure~\ref{X45f2L1LS} show that the error bounds \eqref{peb1}/\eqref{peb2}/\eqref{peb3}, 
\eqref{lsb} and \eqref{peb123s1} of Sections~\ref{ell1} and \ref{ell2} are useful if $\mu\le q$. In the case $\mu> q$, when negative powers of 
$\|\x_j-\z\|_2$ are present in the estimates, they seem too coarse.

\end{enumerate}


\red{
\section{Possible Applications}
Apart from providing consistency analysis of the polynomial generalized finite difference methods, the results of this paper may be useful for the actual design of such
methods in a variety of ways. The following two points seem particularly promising.
\begin{itemize}
\item Our error bounds and numerical experiments suggest that particular weights are preferable for the numerical differentiation 
with a given polynomial exactness order $q$. Indeed, if  $f$ is sufficiently smooth to belong to the Sobolev space $W^q_\infty$ in a domain containing the
convex hull of $\{\z\}\cup\X$, then the $\dnorm_{1,q}$-minimal formula gives the best possible bound \eqref{peb1a} among all formulas covered by
Theorem~\ref{gft}, whereas the  $\dnorm_{2,q}$-minimal formula possesses an error bound \eqref{lsbwoa} that in view of \eqref{rho12} is worse than \eqref{peb1a} at most 
by the factor $\sqrt{N}$. Numerical examples in Section~\ref{expmu} also support the choice $\mu=q$ among $\dnorm_{1,\mu}$-minimal and  $\dnorm_{2,\mu}$-minimal
formulas.
\item Growth functions $\rho_{q,D}(\z,\X,1,q)$ and $\rho_{q,D}(\z,\X,2,q)$ that appear in the estimates \eqref{peb1a}, \eqref{peb1s1}, \eqref{lsbwoa} can be computed as
$\rho_{q,D}(\z,\X,1,q)=\|\w^{1,q}\|_{1,q}$ and $\rho_{q,D}(\z,\X,2,q)=\|\w^{2,q}\|_{2,q}$, respectively, and used to assess the accuracy of numerical differentiation
on a given set $\X$, which may be used to improve the selection of sets of influence in the generalized finite difference methods. The ability to select good sets of
influence is essential for the design of competitive adaptive algorithms \cite{ODP17}. Moreover, it may be possible to generate 
discretization nodes for a given domain $\Omega$ in such a way that the local growth functions for a given degree are small on these nodes, which guarantees small
consistency error. 
\end{itemize}
The same considerations apply to any other area where accurate numerical differentiation is of interest, e.g.\ scattered data fitting or gradient free optimization.
}

\bibliographystyle{abbrv} %
\bibliography{RSbib,meshless}

\begin{thebibliography}{10}

\bibitem{Adams03}
R.~Adams and J.~Fournier.
\newblock {\em Sobolev Spaces, 2nd Edition}.
\newblock Academic Press, Amsterdam, 2003.

\bibitem{BSW61}
F.~L. Bauer, J.~Stoer, and C.~Witzgall.
\newblock Absolute and monotonic norms.
\newblock {\em Numerische Mathematik}, 3(1):257--264, 1961.

\bibitem{beatson-et-al:2010-1}
R.~Beatson, O.~Davydov, and J.~Levesley.
\newblock Error bounds for anisotropic {RBF} interpolation.
\newblock {\em Journal of Approximation Theory}, 162:512--527, 2010.

\bibitem{Benito03}
J.~Benito, F.~Ureña, L.~Gavete, and R.~Alvarez.
\newblock An h-adaptive method in the generalized finite differences.
\newblock {\em Comput. Methods Appl. Mech. Eng.}, 192(5-6):735--759, 2003.

\bibitem{CMV2005}
M.~Caliari, S.~D. Marchi, and M.~Vianello.
\newblock Bivariate polynomial interpolation on the square at new nodal sets.
\newblock {\em Applied Mathematics and Computation}, 165(2):261 -- 274, 2005.

\bibitem{CR1972}
P.~G. Ciarlet and P.~A. Raviart.
\newblock General lagrange and hermite interpolation in rn with applications to
  finite element methods.
\newblock {\em Archive for Rational Mechanics and Analysis}, 46(3):177--199,
  1972.

\bibitem{CSV08b}
A.~R. Conn, K.~Scheinberg, and L.~N. Vicente.
\newblock Geometry of sample sets in derivative-free optimization: polynomial
  regression and underdetermined interpolation.
\newblock {\em IMA Journal of Numerical Analysis}, 28(4):721--748, 2008.

\bibitem{Davy02}
O.~Davydov.
\newblock On the approximation power of local least squares polynomials.
\newblock In J.~Levesley, I.~J. Anderson, and J.~C. Mason, editors, {\em
  Algorithms for Approximation IV}, pages 346--353. University of Huddersfield,
  UK, 2002.

\bibitem{davydov:2007-1}
O.~Davydov.
\newblock Error bound for radial basis interpolation in terms of a growth
  function.
\newblock In A.~Cohen, J.~L. Merrien, and L.~L. Schumaker, editors, {\em Curve
  and {S}urface {F}itting: {A}vignon 2006}, pages 121--130. Nashboro Press,
  Brentwood, 2007.

\bibitem{DavyOanh11}
O.~Davydov and D.~T. Oanh.
\newblock Adaptive meshless centres and {RBF} stencils for {P}oisson equation.
\newblock {\em J. Comput. Phys.}, 230:287--304, 2011.

\bibitem{DavyOanh11sp}
O.~Davydov and D.~T. Oanh.
\newblock On the optimal shape parameter for {G}aussian radial basis function
  finite difference approximation of the {P}oisson equation.
\newblock {\em Comput. Math. Appl.}, 62:2143--2161, 2011.

\bibitem{DPR2014}
O.~Davydov, J.~Prasiswa, and U.~Reif.
\newblock Two-stage approximation methods with extended {B}-splines.
\newblock {\em Mathematics of Computation}, 83:809 -- 833, 2014.

\bibitem{DavySchaback16}
O.~Davydov and R.~Schaback.
\newblock Error bounds for kernel-based numerical differentiation.
\newblock {\em Numerische Mathematik}, 132(2):243--269, 2016.

\bibitem{DavySchabackOPT}
O.~Davydov and R.~Schaback.
\newblock Optimal stencils in {S}obolev spaces.
\newblock {\em IMA Journal of Numerical Analysis}, published online 28 December
  2017.

\bibitem{DavyZeil04}
O.~Davydov and F.~Zeilfelder.
\newblock Scattered data fitting by direct extension of local polynomials to
  bivariate splines.
\newblock {\em Adv. Comput. Math.}, 21(3-4):223--271, 2004.

\bibitem{fasshauer:2007-1}
G.~F. Fasshauer.
\newblock {\em Meshfree Approximation Methods with MATLAB}, volume~6 of {\em
  Interdisciplinary Mathematical Sciences}.
\newblock World Scientific Publishers, Singapore, 2007.

\bibitem{FFprimer15}
B.~Fornberg and N.~Flyer.
\newblock {\em A Primer on Radial Basis Functions with Applications to the
  Geosciences}.
\newblock Society for Industrial and Applied Mathematics, Philadelphia, PA,
  USA, 2015.

\bibitem{FR13}
S.~Foucart and H.~Rauhut.
\newblock {\em A Mathematical Introduction to Compressive Sensing}.
\newblock Birkh\"auser, Basel, 2013.

\bibitem{jetter-et-al:1999-1}
K.~Jetter, J.~St\"ockler, and J.~Ward.
\newblock Error estimates for scattered data interpolation on spheres.
\newblock {\em Mathematics of Computation}, 68:733--747, 1999.

\bibitem{LiszOrk80}
T.~Liszka and J.~Orkisz.
\newblock The finite difference method at arbitrary irregular grids and its
  application in applied mechanics.
\newblock {\em Comput. Struct.}, 11:83--95, 1980.

\bibitem{ODP17}
D.~T. Oanh, O.~Davydov, and H.~X. Phu.
\newblock Adaptive {RBF-FD} method for elliptic problems with point
  singularities in {2D}.
\newblock {\em Applied Mathematics and Computation}, 313:474--497, 2017.

\bibitem{Ostermann_et_al13}
I.~Ostermann, J.~Kuhnert, D.~Kolymbas, C.-H. Chen, I.~Polymerou,
  V.~{\v{S}}milauer, C.~Vrettos, and D.~Chen.
\newblock Meshfree generalized finite difference methods in soil
  mechanics---part {I}: theory.
\newblock {\em GEM - International Journal on Geomathematics}, 4(2):167--184,
  2013.

\bibitem{Schaback_error_analysis17}
R.~Schaback.
\newblock Error analysis of nodal meshless methods.
\newblock In M.~Griebel and M.~A. Schweitzer, editors, {\em Meshfree Methods
  for Partial Differential Equations VIII}, pages 117--143. Springer
  International Publishing, 2017.

\bibitem{SeiboldPhD06}
B.~Seibold.
\newblock {\em M-Matrices in Meshless Finite Difference Methods}.
\newblock Dissertation, University of Kaiserslautern, 2006.

\bibitem{Seibold08}
B.~Seibold.
\newblock Minimal positive stencils in meshfree finite difference methods for
  the {P}oisson equation.
\newblock {\em Comput. Methods Appl. Mech. Eng.}, 198(3-4):592--601, 2008.

\bibitem{Seibold10}
B.~Seibold.
\newblock Performance of algebraic multigrid methods for non-symmetric matrices
  arising in particle methods.
\newblock {\em Numerical Linear Algebra with Applications}, 17:433--451, 2010.

\bibitem{Stetter73}
H.~J. Stetter.
\newblock {\em Analysis of Discretization Methods for Ordinary Differential
  Equations}.
\newblock Springer, New York, 1973.

\bibitem{Stewart98}
G.~Stewart.
\newblock {\em Matrix Algorithms}.
\newblock SIAM, 1998.

\bibitem{wendland:2005-1}
H.~Wendland.
\newblock {\em Scattered Data Approximation}.
\newblock Cambridge University Press, 2005.

\end{thebibliography}

\end{document}